\documentclass[10pt, reqno]{article}
\usepackage{geometry}           
\usepackage{amsmath, amsthm, amssymb, titlesec, titletoc}

\newtheorem{theorem}{Theorem}[section]
\newtheorem{corollary}[theorem]{Corollary}
\newtheorem{lemma}[theorem]{Lemma}
\newtheorem{proposition}[theorem]{Proposition}
\theoremstyle{definition}

\newtheorem{remark}[theorem]{Remark}

\titleformat{\section}[block]{\scshape\filcenter}{\thesection.}{3pt}{}
\titleformat{\subsection}[block]{\scshape\filcenter}{\thesubsection.}{3pt}{}
\titlecontents{section}[0pt]{}{\thecontentslabel. \,}{}{\hfill\contentspage}

\title{\large{\textbf{A PLANARITY ESTIMATE FOR PINCHED SOLUTIONS OF MEAN CURVATURE FLOW}}}
\author{\textsc{\small KEATON NAFF}}
\date{}

\numberwithin{equation}{section}

\begin{document}
\maketitle

\abstract{We show that the blow-ups of compact solutions to the mean curvature flow in $\mathbb{R}^N$ initially satisfying the pinching condition $|A|^2 < c |H|^2$ for a suitable constant $c = c(n)$ must be codimension one.}

\section{Introduction}

In this note, we are interested in studying singularity formation in the mean curvature flow in higher codimension. There is little hope to classify singularity models in general, even in codimension one. However, for suitable assumptions on initial data, this problem becomes more feasible. For hypersurfaces, assumptions like mean-convexity, convexity, and two-convexity (\cite{Hui84},\cite{HS99},\cite{HS09}) strongly restrict the types of singularity models that can appear. In higher codimension, the second fundamental form is much more complicated and preserved curvature conditions for the mean curvature flow are, so far, relatively rare. Here, we are interested in a preserved curvature pinching condition discovered by Andrews and Baker in \cite{AB10}. Research in this setting thus far \cite{AB10},\cite{Ngu18} has suggested that singularity models for these pinched flows might always be codimension one. This behavior is quite interesting, as the original flow may be of arbitrarily high codimension. Colding and Minicozzi \cite{CM20} have recently shown that if the asymptotic shrinker of a singularity model is a multiplicity-one cylinder, then the solution must be codimension one. The multiplicity-one assumption is difficult to verify in practice because embedded initial data need not remain embedded in higher codimension, and so higher multiplicity can occur. Here we take an alternative approach and we prove that blow-ups are codimension one directly. 

The condition introduced by Andrews-Baker \cite{AB10} is a curvature pinching inequality $|A|^2 < c|H|^2$. Importantly, they proved this condition is preserved by the mean curvature flow if $c \leq \frac{4}{3n}$. Let us call a solution of the mean curvature flow satisfying $|A|^2 < c|H|^2$ a \textit{$c$-pinched flow}. The dimensional bound $\frac{4}{3n}$ is a technical constraint imposed by the proof that $c$-pinching is preserved. It is more natural to consider $c = \frac{1}{n-k}$ for $k \in \{1, \dots, n-1\}$.  In codimension one, the condition $|H| > 0$ is equivalent to mean-convexity and the condition $|A|^2 < \frac{1}{n-k}|H|^2$ implies $k$-convexity. Note that $\frac{1}{n-k} \leq \frac{4}{3n}$ if $n \geq 4k$. The mean curvature flow of convex and two-convex hypersurfaces has been studied in the foundational works of Huisken \cite{Hui84} and Huisken-Sinestrari \cite{HS99},\cite{HS09}. The results of \cite{AB10} and \cite{Ngu18} are extensions of these results established by Huisken and Sinestrari to higher codimension, assuming the stronger pinching condition. To contextualize our result, we must first discuss these. 

In the seminal paper \cite{Hui84} - which draws upon Hamilton's foundational work on the Ricci flow \cite{Ham82} - Huisken proved the mean curvature flow evolves closed convex hypersurfaces into spherical singularities. Using the techniques developed there - in particular, the delicate Stampacchia iteration - Andrews and Baker proved that the mean curvature flow in $\mathbb{R}^N$ will deform closed $n$-dimensional initial data satisfying 
\[
 |A|^2 < c_{1} |H|^2, \quad c_{1} := \min\Big\{ \frac{4}{3n}, \frac{1}{n-1} \Big\} = \begin{cases} \frac{4}{3n} & \text{if} \;\;n = 2, 3 \\ \frac{1}{n-1} & \text{if}\;\; n \geq 4 \end{cases}
\]
to a point in finite time. Moreover, the flow is asymptotic to a family of shrinking round spheres contained in some $(n+1)$-dimensional affine subspace of $\mathbb{R}^N$.

Because $\frac{4}{3n} \leq \frac{1}{n-1}$ for $n \leq 4$ and $c$-pinching is only preserved for $c \leq \frac{4}{3n}$, the previous result is the best currently possible if $n \leq 4$. Suppose now we have a closed initial submanifold satisfying the weaker pinching condition studied by Nguyen \cite{Ngu18}:
\[
|A|^2 < c_2|H|^2, \quad c_2 := \min\Big\{ \frac{4}{3n}, \frac{1}{n-2} \Big\} = \begin{cases} \frac{4}{3n} & \text{if} \;\;n = 5, 6, 7 \\ \frac{1}{n-2} & \text{if}\;\; n \geq 8 \end{cases}.
\]
In \cite{HS99} and \cite{HS09}, Huisken and Sinestrari established several very important a priori estimates for the study of flows of two-convex hypersurfaces. In particular, they establish a \textit{convexity estimate} for mean-convex flows, which shows blow-ups must be weakly convex, and a \textit{cylindrical estimate} for two-convex flows, which shows blow-ups must satisfy $|A|^2 \leq \frac{1}{n-1} H^2$. Also very important in their work is a \textit{pointwise derivative estimate} for the second fundamental form.  In higher codimension, we no longer have a notion of convexity, but the derivative and cylindrical estimates still make sense. In \cite{Ngu18}, Nguyen first proves a pointwise derivative estimate using the pinching condition, and then uses a blow-up argument to establish the cylindrical estimate for $c_2$-pinched flows. Specifically, his quantitative cylindrical estimate shows the following alternative: for a $c_2$-pinched flow, either the flow becomes everywhere $c_1$-pinched, hence spherical, or there are regions of the manifold which are becoming arbitrarily close to a codimension one round cylinder $S^{n-1} \times \mathbb{R}$ contained in $\mathbb{R}^{N}$ prior to the first singular time.

The result of \cite{Ngu18} leaves open the possibility that a potential ``cap" to a forming cylindrical singularity may not lie in the same $(n+1)$-dimensional subspace as the ``neck''. Presently, we will show that this does not occur and more generally that all high curvature regions prior to singularity formation must become nearly planar. Results of this type have been obtained before in the curve shortening flow. In \cite{Alt91}, Altschuler proved singularities of the curve shortening flow in $\mathbb{R}^3$ are always planar. 

Here is our setting. We suppose $n \geq 5$ and our initial data $M_0 = F_0(M)\subset \mathbb{R}^N$ is an $n$-dimensional, closed, immersed submanifold satisfying 
\[
|A|^2 < c_n|H|^2\quad c_n := \min\Big\{\frac{4}{3n}, \frac{3(n+1)}{2n(n+2)} \Big\} =  \begin{cases} \frac{3(n+1)}{2n(n+2)} & \text{if} \;\; n = 5, 6, \text{ or } 7 \\ \frac{4}{3n} &  \text{if} \;\; n \geq 8 \end{cases}.
\]
For $n = 5$ and $n =6$, the constant $c_n$ is strictly between $\frac{1}{n-1}$ and $\frac{4}{3n}$. This value of $c_n$ in these dimensions is the largest we can allow in our new estimates in the proof of our main theorem below. For $n \geq 7$, we have $\frac{3(n+1)}{2n(n+2)} \geq \frac{4}{3n}$, with equality when $n = 7$. The value of $c_n$ in higher dimensions is the largest allowed by estimates used in \cite{AB10} to prove that  $c$-pinching is preserved. We will use these estimates as well. Under these assumptions, we consider a maximal solution $M_t = F(M,t) \subset \mathbb{R}^N$, $t\in [0, T)$, to the mean curvature flow, where $T$ is the first singular time. 

For the purpose of studying planarity, we define a tensor $\hat A$ by
\begin{align}\label{def-hat-A}
\hat A(X, Y) := A(X, Y) - \frac{\langle A(X, Y), H \rangle}{|H|^2}  H,
\end{align}
for vector fields  $X$ and $Y$ tangent to $M_t$. Since $c_n |H|^2 > |A|^2 \implies |H| > 0$ for $t \in [0, T)$, this tensor is well-defined along the flow. Under suitable assumptions, $\hat A$ vanishes identically if and only if our submanifold is a hypersurface inside an $(n+1)$-dimensional affine subspace of $\mathbb{R}^N$. See Proposition \ref{whyhatA} in Section 2. Here is our main theorem:

\begin{theorem}\label{main}
Suppose $n \geq 5$ and $N > n+1$. Let $c_n = \min\{\frac{4}{3n}, \frac{3(n+1)}{2n(n+2)}\}$. Suppose $M_t = F(M,t) \subset \mathbb{R}^N$, $t \in [0, T)$, is a smooth family of $n$-dimensional, closed, immersed submanifolds evolving by mean curvature flow, which initially satisfies $|A|^2 < c_n |H|^2$. Then there are constants $\sigma > 0$ and $C < \infty$, depending only upon the initial submanifold $M_0$, such that
\[
|\hat A|^2 \leq C|H|^{2 -\sigma}  
\]
holds pointwise on $M \times [0, T)$.  
\end{theorem}

Together with Proposition \ref{whyhatA} below, this result shows that at the first singular time, blow-ups must be codimension one. We show this in Proposition \ref{singularities}. In particular, an estimate for $|\hat A|/|H|$ is a measure of how far our submanifold is from being planar. The result above shows that the mean curvature flow preserves near-planarity under the $c_n$-pinching assumption. Since $\frac{1}{n-2} \leq \frac{4}{3n}$ for $n \geq 8$, the $c_2$-pinching condition considered in \cite{Ngu18} is included in the theorem above, at least when $n \geq 8$. Our result also applies for weaker pinching constants of the form $c = \frac{1}{n-k}$ if $n \geq 4k$. These weaker pinching constraints will allow a wider range of singularities models and our result shows these must also be codimension one.

Let $\tilde c_2 := \min\{\frac{3(n+1)}{2n(n+2)}, \frac{1}{n-2}\} = \min\{c_2, \frac{3(n+1)}{2n(n+2)}\}$. For $\tilde c_2$-pinched flows, the planarity estimate and Nguyen's cylindrical and derivative estimates all hold. Consequently, for $\tilde c_2$-pinched flows, using a recent classification result of Brendle-Choi \cite{BC19},\cite{BC20}, it is possible to give a complete classification of singularity models at the first singular time. The pinching constant and, consequently, the classification are sharp when $n \geq 8$. We can prove the following corollary by establishing the noncollapsing of singularity models. The proof is an adaptation of the work of Huisken-Sinestrari \cite{HS09}. The details can be found in \cite{Naf19}.

\begin{corollary}
Suppose $n \geq 5$ and $N > n+1$. Let $\tilde c_2 = \min\{\frac{3(n+1)}{2n(n+2)}, \frac{1}{n-2}\}$. Consider a family of closed $n$-dimensional submanifolds in $\mathbb{R}^{N}$ evolving by mean curvature flow which initially satisfy $|A|^2 < \tilde{c}_2 |H|^2$. At the first singular time, the only possible blow-up limits are codimension one shrinking round spheres, shrinking round cylinders, and translating bowl solitons. 
\end{corollary}

The structure of this article is as follows. In Section 2, we record useful notation and standard identities for the higher codimension mean curvature flow. We also show that if $\hat A$ vanishes, then the submanifold is codimension one. In Section 3, we derive the evolution equation for $|\hat A|^2$. In Section 4, we prove Theorem 1.1 via the maximum principle. 

There is a connection between the mean curvature flow of convex and two-convex hypersurfaces and the Ricci flow of initial data with positive isotropic curvature (PIC). Positive isotropic curvature was introduced by Micallef and Moore \cite{MM88} for the study of minimal two-spheres and has been studied in the Ricci flow since Hamilton's fundamental paper \cite{Ham97}. If a submanifold $M \subset \mathbb{R}^N$ satisfies $|A|^2 < \frac{1}{n-2}|H|^2$, or if $M$ is a two-convex hypersurface, then the induced metric on $M$ has positive isotropic curvature. Consequently, if $M$ satisfies $|A|^2 < \frac{1}{n-1}|H|^2$, or if $M$ is a convex hypersurface, then the induced metric on $M \times \mathbb{R}$ has positive isotropic curvature (this property is called PIC1). In \cite{Bre08}, Brendle showed that the Ricci flow evolves PIC1 initial data into round spheres. As for PIC initial data, Hamilton's breakthrough in \cite{Ham97} was to show that in dimension four the Ricci flow of PIC initial data only develops neck-pinch singularities. The study of PIC initial data for the Ricci flow in higher dimensions $(n \geq 12)$ has recently been addressed by Brendle in \cite{Bre19}. Although the results for the Ricci flow and the mean curvature flow in these contexts are very similar, it is interesting that the proofs have tended to be somewhat different.

\textbf{Acknowledgements.} I am very grateful to my advisor, Simon Brendle, for his guidance and for helpful comments on an earlier version of this paper. I would also like to thank the anonymous referees for their comments, which have improved the clarity of this paper. This project was supported by the National Science Foundation under grant DMS-1649174. 

\section{Notation and Preliminaries on the MCF in Higher Codimension} 
In this section we will record notation, identities, and results that we will use in the proof of our theorem. We suppose we are given a solution of mean curvature flow $M_t = F(M,t) \subset \mathbb{R}^N$, where $M$ is an abstract manifold and $F: M \times [0, T) \to \mathbb{R}^N$ is a smooth familiy of parametrizations satisfying $\partial_t F = H$. Here $H$ denotes the mean curvature vector, while $g$ and $A$ denote the metric and the second fundamental form. We let $TM$ and $NM$ denote the (time-dependent) tangent and normal bundles of $M$ and recall both bundles are subbundles of $M \times \mathbb{R}^N \cong F^\ast T\mathbb{R}^N$. 

We let $D$ denote the Euclidean derivative on $F^\ast T\mathbb{R}^N$ and let $\nabla_X Y= (D_XY)^\top$ and $\nabla^\perp_X \nu = (D_X \nu)^\perp$ denote the induced connections on $TM$ and $NM$. It is possible to view the second fundamental form as either a section of $T^\ast M \otimes T^\ast M \otimes NM$ or a section of $T^\ast M \otimes T^\ast M \otimes F^\ast T\mathbb{R}^N$. By a minor abuse of the notation, we let $\nabla^\perp$ denote the induced connection on the former bundle and $\nabla$ denote the induced connection on the latter bundle. We can similarly view $H$ in either $NM$ or $F^\ast T \mathbb{R}^N$, and so we distinguish $\nabla^\perp H$ and $D H$. We will do our computations in local coordinates on $M$. For a fixed point $p_0 \in M$ and a fixed time $t_0 \in [0, T)$, we consider normal coordinates $(x_1, \dots, x_n)$ around $p$ such that $e_i = \partial_{x_i}|_{p_0}$ is an orthonormal basis of $T_{p_0}M$ and $\nabla_i \, \partial_{x_j}|_{p_0} = \nabla_{e_i} \partial_{x_j}|_{p_0} = 0$ at time $t_0$. We will use latin indices $i, j, k, \dots$ to indicate tangential components of tensors. We will not make use of the natural space-time connection for bundles over $M \times [0, T)$, although one certainly could.

We use Einstein summation notation: repeated Latin indices in multiplied tensor components will by default indicate summation from $1$ to $n$. Sometimes we will include the summation symbol to emphasize its presence. Since we work with an orthonormal basis we will raise and lower indices freely (except those of the metric tensor). For example, for $(0,2)$-tensors $T$ and $S$, we allow ourselves to write
\[
T_{ik} S_{jk} = T_{ik} S_j{}^k= g^{kl} T_{ik} S_{jl} = \sum_{k =1}^n T_{ik} S_{jk}. 
\]
Since we do not use a covariant time derivative, it is important to keep track of the metric when differentiating in time. We recall that in higher codimension the evolution equations for the metric and its inverse are
\begin{align}
 \frac{\partial}{\partial t} g_{ij} &= -2 \langle A_{ij}, H \rangle, \\
\label{evolution-of-inverse} \frac{\partial}{\partial t} g^{ij} &= 2 \langle A^{ij}, H \rangle = 2\langle A_{ij}, H \rangle.  
\end{align}
Sometimes, such as above, we will not use indices for the components of tensors taking values in the normal bundle. We will use the inner product $\langle \cdot\, , \cdot \rangle$ to indicate summation over normal directions. Other times, we will use Greek indices $\alpha, \beta, \gamma, \dots$ to indicate normal components of tensors. In these instances, repeated Greek indices will \textit{usually} indicate summation from $1$ to $N-n$. However, often we will only sum from $2$ to $N-n$ and in these cases, we will include a summation symbol to emphasize that.  To illustrate our convention, suppose $\nu_1, \dots, \nu_{N-n}$ is local orthonormal frame for the normal bundle. Then $A_{ij\alpha} = \langle A(e_i, e_j), \nu_\alpha\rangle$ and 
\[
|\langle A_{ij}, A_{kl} \rangle|^2 = \bigg| \sum_{\alpha = 1}^{N-n} A_{ij\alpha} A_{kl\alpha}\bigg|^2 = \langle A_{ij}, A_{kl} \rangle\langle A_{ij}, A_{kl} \rangle = \sum_{i,j,k,l =1}^n \sum_{\alpha, \beta = 1}^{N - n} A_{ij\alpha} A_{kl\alpha}A_{ij\beta}A_{kl\beta}.
\]
While the meaning of the left-hand side is perhaps a bit less obvious, it has the advantage making our computation much more succinct. For the norm of traced tensors, summation will always take place inside the norm. For example, 
\[
|\langle A_{ik}, A_{jk} \rangle |^2 = \bigg|\sum_{k=1}^n \langle A_{ik}, A_{jk} \rangle\bigg|^2 =  \langle A_{ik}, A_{jk} \rangle \langle A_{il}, A_{jl} \rangle = \sum_{i,j,k,l =1}^n \langle A_{ik}, A_{jk} \rangle \langle A_{il}, A_{jl} \rangle.
\]

We call $|H|$ the \textit{scalar mean curvature}. Since we assume $c_n|H|^2 > |A|^2 \implies |H| >  0$, we can define 
\begin{equation}
\label{principal-normal} \nu_1 := \frac{H}{|H|}
\end{equation}
to be the \textit{principal normal direction}. Of course, $ |\nu_1| = 1$. From now on, $\nu_1$ will always denote the vector defined in \eqref{principal-normal} and $\nu_2, \dots, \nu_{N-n}$ will denote a local frame of normal vectors orthogonal to $\nu_1$. Having defined the principal direction, the $(0,2)$-tensor 
\begin{equation}
h_{ij} := \langle A_{ij}, \nu_1 \rangle = A_{ij1}
\end{equation}
 is the \textit{principal component of the second fundamental form}. This is the only nonzero component of the second fundamental form if our submanifold happens to be codimension one. 
 With this notation, we can express $\hat A$, defined in \eqref{def-hat-A}, as 
\begin{equation} 
 \hat A_{ij} = \sum_{\alpha =2}^{N-n} A_{ij\alpha}\nu_\alpha.
\end{equation}
Note $\hat A_{ij\alpha} = A_{ij\alpha}$ for $\alpha \geq 2$. Moreover, we have the identities
\begin{align}
\label{decomposition-of-A} A_{ij} &= \hat A_{ij} + h_{ij} \nu_1 = \hat A_{ij} + \mathring h_{ij} \nu_1 + \frac{1}{n} |H| g_{ij} \nu_1, \\
 H & = |H| \nu_1, \\
|A|^2 &= |\hat A|^2 + |\mathring h|^2 + \frac{1}{n} |H|^2.
\end{align}
We will often use that $g^{ij}\hat A_{ij} = 0$ as well as the obvious orthogonality relations
\begin{equation}
\label{orthgonality-relations} \langle \hat A_{ij}, \nu_1 \rangle = \langle \nabla_k^\perp \nu_1, \nu_1 \rangle = 0. 
\end{equation}
 \begin{remark}\label{notation-difference}
Unfortunately, the notation used here, in \cite{AB10}, and in \cite{Ngu18} is slightly different. For the reader who has \cite{AB10} or \cite{Ngu18} on hand, here is how to translate: the full second fundamental form, denote by $A$ in this paper, is denoted by $h$ in \cite{AB10} and $A$ in \cite{Ngu18}. The principal component of the second fundamental form, denoted by $h$ in this paper, is denoted by $h_1$ in \cite{AB10} and $A_1$ in \cite{Ngu18}. Finally, the tensor $\hat A$ in this paper is denoted by $h_-$ in \cite{AB10} and $A_-$ in \cite{Ngu18}.  
\end{remark}

The curvature and normal curvature are denoted by $R$ and $R^\perp$ respectively, and our sign convention is that 
\begin{align*}
R(X,Y)Z &= \nabla_Y \nabla_X Z -  \nabla_X \nabla_Y Z - \nabla_{[Y, X]} Z, \\
R^\perp(X, Y)\nu &= \nabla^\perp_Y \nabla^\perp_X \nu -  \nabla^\perp_X \nabla^\perp_Y \nu - \nabla^\perp_{[Y, X]} \nu. 
\end{align*}
In higher codimension, the fundamental Gauss, Codazzi, and Ricci equations in Euclidean space in a local frame take the form
\begin{align}
\label{gauss} R_{ijkl} &= \langle A_{ik}, A_{jl} \rangle - \langle A_{il}, A_{jk} \rangle, \\
\label{codazzi} \nabla_i^\perp A_{jk} &= \nabla_j^\perp A_{ik}, \\
\label{ricci} R^\perp_{ij\alpha\beta} &= \sum_{k=1}^n (A_{ik\alpha} A_{jk\beta}  - A_{jk\alpha} A_{ik\beta}).  
\end{align}
We also define a vector-valued version of the normal curvature by 
\begin{equation*}
R^\perp_{ij}(\nu_\alpha) = R^\perp_{ij\alpha\beta} \nu_\beta = A_{ik\alpha} A_{jk} - A_{jk\alpha} A_{ik}.
\end{equation*}
In particular, we note that $R^\perp_{ij}(\nu_1) = h_{ik}A_{jk} - h_{jk}A_{ik}$, which in view of \eqref{decomposition-of-A} gives
\begin{equation}
\label{normal-curvature-in-principal-direction} R^\perp_{ij}(\nu_1) = \mathring h_{ik} \hat A_{jk} - \mathring h_{jk}\hat A_{ik}. 
\end{equation}
For $\alpha \geq 2$, we have 
\begin{align*}
R^\perp_{ij}(\nu_{\alpha}) & = \hat A_{ik\alpha}(\mathring h_{jk}\nu_1 + \hat A_{jk}) - \hat A_{jk\alpha}(\mathring h_{ik} \nu_1 + \hat A_{jk}) \\
& = -\langle R^\perp_{ij}(\nu_1), \nu_{\alpha}\rangle \nu_1 + \hat A_{ik\alpha} \hat A_{jk} - \hat A_{jk\alpha}\hat A_{ik}.
\end{align*} 
To summarize, for $\alpha, \beta \in \{1, \dots, N-n\}$, we have
\begin{equation}
\label{normal-curvature-components} R^\perp_{ij\alpha\beta} = 
\begin{cases} 
0 & \alpha = \beta \\
\mathring h_{ik} \hat A_{jk\beta} - \mathring h_{jk}\hat A_{ik\beta} & \alpha = 1,\; \beta \geq 2\\
\mathring h_{jk} \hat A_{ik\alpha} - \mathring h_{ik}\hat A_{jk\alpha}  & \alpha \geq 2, \;\beta = 1 \\
\hat A_{ik\alpha} \hat A_{jk\beta}  - \hat A_{jk\alpha} \hat A_{ik\beta} & \alpha, \, \beta\geq 2
\end{cases}.
\end{equation}

We define a new connection for the orthogonal decomposition of $NM = E_1 \oplus \hat E$ where $\hat E$ consists of normal vectors $\hat \nu$ which are everywhere orthogonal to $\nu_1$, $\langle \hat \nu, \nu_1 \rangle = 0$, and $E_1 = C^\infty(M)\nu_1$. Define $\hat \nabla^{\perp}$ on $\hat E$ by 
\begin{equation}
\hat \nabla_i^\perp  \hat \nu := \nabla_i^\perp \hat \nu - \langle \nabla_i^\perp \hat \nu, \nu_1 \rangle \nu_1.
\end{equation}
Since, by definition, $\hat A$ is a section of $T^\ast M\otimes T^\ast M \otimes \hat E$, it is natural to define the connection $\hat \nabla^\perp$ on $\hat A$ by 
\begin{equation}
\hat \nabla_i^\perp\hat A_{jk} := \nabla_i^\perp \hat A_{jk} - \langle \nabla_i^\perp \hat A_{jk}, \nu_1 \rangle \nu_1. 
\end{equation}
For clarity, let us summarize the relationship between the derivatives $D,\nabla, \nabla^\perp$, and $\hat \nabla^\perp$ acting on $A, H$, and $\hat A$. If we view $A_{jk}$ and $H$ as taking values in $F^\ast T\mathbb{R}^N$, then we can decompose $\nabla_i A_{jk}$ and $D_i H$ into the tangential and normal components to get
\begin{align}
\nonumber \nabla_i A_{jk} &= (\nabla_i A_{jk})^\perp + (\nabla_i A_{jk})^\top \\
 & = \nabla_i^\perp A_{jk} - \sum_{l =1}^n \langle A_{jk}, A_{il} \rangle e_l, \\
\nonumber D_i H &= (D_i H)^\perp + (D_i H)^\top\\
\label{euclidean-deriviative-H}  & = \nabla_i^\perp H - \sum_{j = 1}^n \langle H, A_{ij} \rangle e_j. 
\end{align}
Similarly, and more relevant to the coming computations, if we view $A_{jk}$, $H$, and $\hat A_{jk}$ as taking values in the normal bundle, then we can decompose $\nabla_i^\perp A_{jk}$, $\nabla^\perp_i H$, and $\nabla^\perp_i \hat A_{jk}$ using $NM = E_1 \oplus \hat E$ to get 
\begin{align}
\label{derivative-of-A}\nabla_i^\perp A_{jk} & = \big(\hat \nabla_i^\perp \hat A_{jk} + h_{jk}\nabla_i^\perp \nu_1\big) + \big(\langle \nabla_i^\perp \hat A_{jk}, \nu_1 \rangle + \nabla_i h_{jk}\big)\nu_1,\\
\label{derivative-of-H}\nabla_i^\perp H &= |H| \nabla_i^\perp \nu_1+ (\nabla_i |H|) \nu_1,\\
\nabla_i^\perp \hat A_{jk} &= \hat \nabla_i^\perp \hat A_{jk} + \langle \nabla_i^\perp \hat A_{jk}, \nu_1 \rangle \nu_1.
\end{align}
Note that $\nabla_i^\perp A_{jk}$ is shorthand for $(\nabla_i^\perp A_{jk\alpha})\nu_\alpha$. As a direct consequence of the orthogonality relations, we have
\begin{proposition}[Decomposition of derivatives]\label{decompgrad}
\begin{align}
\label{norm-of-derivative-of-A} |\nabla^\perp A|^2& = |\hat \nabla_i^\perp \hat A_{jk} + h_{jk}\nabla_i^\perp \nu_1|^2 + |\langle \nabla_i^\perp \hat A_{jk}, \nu_1 \rangle + \nabla_i h_{jk}|^2.\\
\label{nabla-perp-H-squared-formula}|\nabla^\perp H|^2 &= |H|^2 |\nabla^\perp \nu_1|^2+ |\nabla |H||^2.\\
\label{norm-of-derivative-of-hat-A} |\nabla^\perp \hat A|^2&= |\hat \nabla^\perp \hat A|^2 + |\langle \nabla^\perp \hat A, \nu_1 \rangle|^2.
\end{align}
\end{proposition}
\noindent We will use these identities in Sections 3 and 4.  

It is very useful to consider the implications of the Codazzi equation for the decomposition of $\nabla_i^\perp A_{jk}$ above. Projecting the Codazzi equation \eqref{codazzi} onto $E_1$ and $\hat E$ implies the both of the tensors
\begin{equation*}
\nabla_i h_{jk} + \langle \nabla^\perp_i \hat A_{jk}, \nu_1 \rangle\; \text{ and }\;\; \hat \nabla^\perp_i \hat A_{jk} + h_{jk} \nabla_i^\perp \nu_1
\end{equation*}
are symmetric in $i, j, k$. Consequently, it is equivalent to trace over $j, k$ or trace over $i, k$, and this implies
\begin{align}
\label{codazzi-1}\sum_{k=1}^n\nabla_k h_{ik} + \langle \nabla^\perp_k \hat A_{ik} , \nu_1 \rangle &= \nabla_i |H|, \\
\label{codazzi-perp}\sum_{k=1}^n\hat \nabla_k^\perp \hat A_{ik} + h_{ik} \nabla_k^\perp \nu_1 &=  |H| \nabla_i^\perp \nu_1.
\end{align}

Next, we review the evolution equations for $A$ and $H$ in higher codimension. We let $\frac{\partial}{\partial t}^\perp$ denote the projection of the time derivative onto the normal bundle and $\Delta^\perp$ denote the Laplacian with respect to the connection $\nabla^\perp$. 
\begin{proposition}[Evolution of $A$ and $H$]\label{evolution-equations}
With the summation convention, the evolution equations of $A_{ij}$ and $H$ are
\begin{align}
\label{evolution-of-A} \Big(\frac{\partial}{\partial t}^\perp - \Delta^\perp\Big)A_{ij} &=  - \langle H, A_{ik} \rangle A_{jk} - \langle H, A_{jk} \rangle A_{ik} +  \langle A_{ij},  A_{kl} \rangle A_{kl} \\
&\qquad - 2\langle A_{ik}, A_{jl} \rangle A_{kl} + \langle A_{ik}, A_{kl} \rangle A_{jl} + \langle A_{jl},  A_{kl} \rangle A_{ik}, \nonumber\\
\label{evolution-of-H} \Big(\frac{\partial}{\partial t}^\perp - \Delta^\perp\Big) H &=  \langle H, A_{kl} \rangle A_{kl}.
\end{align} 
The evolution equations of $|A|^2$ and $|H|^2$ are
\begin{align}
\label{evolution-of-|A|^2} \frac{\partial}{\partial t} |A|^2 &= \Delta|A|^2 - 2|\nabla^\perp A|^2 +  2|\langle A_{ij}, A_{kl} \rangle|^2 + 2|R^\perp|^2,\\
\label{evolution-of-|H|^2}\frac{\partial}{\partial t} |H|^2 &=\Delta |H|^2 - 2|\nabla^\perp H|^2 + 2 |\langle A_{ij}, H \rangle|^2.
\end{align}
\end{proposition}
\begin{proof}
The equations above have been derived in both \cite{Smo12} and \cite{AB10}, but under slightly different notation and conventions (see also Remark \ref{notation-difference}). For the convenience of the reader, we will give a derivation here.
 
The first step is to prove the important Simons' identity, by commuting second derivatives of the second fundamental form. To that end, by standard commutation identities we have
\[
\nabla^\perp_i \nabla^\perp_k A_{jl\alpha} = \nabla^\perp_k \nabla^\perp_i A_{jl\alpha} + R_{ikjp} A_{pl\alpha} + R_{iklp} A_{jp\alpha} + R^\perp_{ik\alpha\beta} A_{jl\beta}.
\]
Plugging in the Gauss equation \eqref{gauss} and Ricci equation \eqref{ricci}, we obtain
\begin{align*}
\nabla^\perp_i \nabla^\perp_k A_{jl\alpha} &= \nabla^\perp_k \nabla^\perp_i A_{jl\alpha} + \big(\langle A_{ij}, A_{kp} \rangle - \langle A_{ip}, A_{jk} \rangle\big) A_{pl\alpha}\\
& \quad +\big (\langle A_{il}, A_{kp} \rangle - \langle A_{ip}, A_{kl} \rangle\big)A_{jp\alpha} + \big(A_{ip\alpha}A_{kp\beta} - A_{ip\beta}A_{kp\alpha}\big) A_{jl\beta}. 
\end{align*}
Note that $(A_{ip\alpha}A_{kp\beta} - A_{ip\beta}A_{kp\alpha}) A_{jl\beta} = \langle A_{kp}, A_{jl} \rangle A_{ip\alpha} - \langle A_{ip}, A_{jl} \rangle A_{kp\alpha}$, so we may write
\begin{align*}
\nabla^\perp_i \nabla^\perp_k A_{jl}& =  \nabla^\perp_k \nabla^\perp_i A_{jl} + \big(\langle A_{ij}, A_{kp} \rangle - \langle A_{ip}, A_{jk} \rangle\big) A_{pl}\\
& \quad + \big(\langle A_{il}, A_{kp} \rangle - \langle A_{ip}, A_{kl} \rangle\big)A_{jp} + \langle A_{kp}, A_{jl} \rangle A_{ip} - \langle A_{ip}, A_{jl} \rangle A_{kp}.
\end{align*}
Using the Codazzi equation \eqref{codazzi} twice gives 
\begin{align*}
\nabla^\perp_i\nabla^\perp_j A_{kl}  &= \nabla^\perp_i \nabla^\perp_k A_{jl}\\
 & =  \nabla^\perp_k \nabla^\perp_l A_{ij} + \big(\langle A_{ij}, A_{kp} \rangle - \langle A_{ip}, A_{jk} \rangle\big) A_{pl}\\
& \quad + \big(\langle A_{il}, A_{kp} \rangle - \langle A_{ip}, A_{kl} \rangle\big)A_{jp} + \langle A_{kp}, A_{jl} \rangle A_{ip} - \langle A_{ip}, A_{jl} \rangle A_{kp}.
\end{align*}
Finally, if we trace over the indices $k$ and $l$, gather terms, and rearrange a bit, we obtain
\begin{align}
\label{Delta-A} \nabla^\perp_i\nabla^\perp_j H  & =  \Delta^\perp A_{ij} - \langle H, A_{ip}\rangle A_{jp}  + \langle A_{ij}, A_{pq} \rangle A_{pq}  \\
\nonumber & \quad - 2\langle A_{ip}, A_{jq} \rangle A_{pq} + \langle A_{iq},  A_{pq} \rangle A_{jp}+ \langle A_{jq}, A_{pq} \rangle A_{ip}.
\end{align}

Next, we compute $\partial_t^\perp A_{ij}$. To do so, let us denote $F_i = dF(\partial_{x_i})$ and recall that $\partial_t F = H$. In our notation we have $\partial_t F_i = D_i H$ and $[\partial_t, D_i] = 0$. Then
\[
\frac{\partial}{\partial t}^\perp A_{ij} = \frac{\partial}{\partial t}^\perp (D_{i} F_j)^\perp = \bigg( \frac{\partial}{\partial t} \Big( D_{i} F_j - g^{pq}\langle D_{i} F_j, F_p \rangle F_q \Big)\bigg)^\perp=  \big(D_{i}D_j H \big)^\perp, 
\]
where we have used that $D_{i}F_j = 0$ at the origin in our coordinates. Now by \eqref{euclidean-deriviative-H} we have $D_jH = \nabla^\perp_j H - \langle H,  A_{jp} \rangle F_p$. Differentiating and taking the normal projection, we obtain
\begin{align}
\label{partial-t-A} \frac{\partial}{\partial t}^\perp A_{ij} = \nabla^\perp_i \nabla^\perp_j H - \langle H, A_{jp}\rangle A_{ip}. 
\end{align}
Combining \eqref{Delta-A} and \eqref{partial-t-A} proves the claimed evolution equation for $A$. The evolution equation for $H$ is a straightforward consequence of the evolution equation for $A_{ij}$ and the identity
\[
\Big(\frac{\partial}{\partial t}^\perp - \Delta^\perp\Big)H = \Big(\frac{\partial}{\partial t} g^{ij}\Big)H + g^{ij} \Big(\frac{\partial}{\partial t}^\perp - \Delta^\perp\Big)A_{ij}, \\
\]
in view of \eqref{evolution-of-inverse}.

To obtain the evolution equation for $|A|^2 = g^{ij}g^{kl} \langle A_{ij}, A_{kl} \rangle$, we use \eqref{evolution-of-inverse} once more to get 
\[
\Big(\frac{\partial}{\partial t} - \Delta\Big)|A|^2 + 2|\nabla^\perp A|^2 = 2\Big\langle \Big(\frac{\partial}{\partial t}^\perp - \Delta^\perp\Big) A_{ij}, A_{ij} \Big\rangle  + 8\langle H, A_{ij} \rangle \langle H, A_{ij} \rangle. 
\]
Plugging in the evolution equation for $A_{ij}$, the terms of the form $|\langle H, A \rangle|^2$ cancel. Relabeling and organizing the indices a bit, we are left with 
\[
\Big(\frac{\partial}{\partial t} - \Delta\Big)|A|^2 + 2|\nabla^\perp A|^2 = 2|\langle A_{ij}, A_{kl} \rangle|^2 + 4\langle A_{ik}, A_{kl} \rangle \langle A_{ij}, A_{jl}\rangle - 4\langle A_{ik}, A_{jl} \rangle \langle A_{ij}, A_{kl}\rangle. 
\]
By the Ricci equation \eqref{ricci}, we have 
\begin{align*}
|R^\perp|^2 &= R^\perp_{ij\alpha\beta}R^\perp_{ij\alpha\beta} \\
& = \big(A_{ik\alpha}A_{jk\beta} - A_{jk\alpha}A_{ik\beta}\big)\big(A_{il\alpha}A_{jl\beta} - A_{jl\alpha}A_{il\beta}\big)\\
&= \langle A_{ik}, A_{il} \rangle \langle A_{jk},A_{jl} \rangle + \langle A_{jk}, A_{jl} \rangle \langle A_{ik}, A_{il} \rangle\\
& \quad - \langle A_{jk},A_{il}\rangle \langle A_{ik}, A_{jl}\rangle - \langle A_{ik}, A_{jl} \rangle \langle A_{jk}, A_{il} \rangle.
\end{align*}
After another relabeling of indices, we see that 
\[
2|R^\perp|^2 = 4\langle A_{ik}, A_{kl} \rangle \langle A_{ij}, A_{jl}\rangle - 4\langle A_{ik}, A_{jl} \rangle \langle A_{ij}, A_{kl}\rangle,
\]
which establishes the evolution equation for $|A|^2$. Finally, the evolution equation of $|H|^2$ follows easily from the evolution equation of $H$. 
\end{proof}

It will be useful to expand each of the reaction terms 
\begin{align*}
|\langle A_{ij}, H \rangle|^2 &= \langle A_{ij}, H \rangle\langle A_{ij}, H \rangle, \\
|\langle A_{ij}, A_{kl} \rangle|^2 &= \langle A_{ij}, A_{kl} \rangle \langle A_{ij}, A_{kl} \rangle,\\
|R^\perp_{ij}|^2 & = R^\perp_{ij\alpha\beta} R^\perp_{ij\alpha\beta},
\end{align*}
on the right-hand side of the evolution equations \eqref{evolution-of-|A|^2} and \eqref{evolution-of-|H|^2}, using the formula \eqref{decomposition-of-A}. It is straightforward to see that 
\begin{equation}
\label{decomp-1} \langle A_{ij}, H \rangle =|H| h_{ij} = \frac{1}{n}|H|^2g_{ij} + |H|\mathring h_{ij}.
\end{equation}
Similarly, 
\begin{align}
\nonumber \langle A_{ij}, A_{kl} \rangle & = h_{ij} h_{kl} + \langle \hat A_{ij}, \hat A_{kl} \rangle\\
\label{decomp-2} & = \frac{1}{n^2}|H|^2 g_{ij}g_{kl} + \frac{1}{n}|H|(g_{ij} \mathring h_{kl} +\mathring h_{ij} g_{kl} ) + \mathring h_{ij} \mathring h_{kl} + \langle \hat A_{ij}, \hat A_{kl} \rangle. 
\end{align}
As for the third reaction term, recalling \eqref{normal-curvature-components}, we have
\begin{align}
\nonumber |R^\perp|^2 &= \sum_{\beta =1}^{N-n} R^\perp_{ij1\beta} R^\perp_{ij1\beta} + \sum_{\alpha =1}^{N-n} R^\perp_{ij\alpha1} R^\perp_{ij\alpha1} + \sum_{\alpha, \beta = 2}^{N-n} R^\perp_{ij\alpha\beta}R^\perp_{ij\alpha\beta} \\
\label{decomp-3}& = 2|R^\perp_{ij}(\nu_1)|^2 +\sum_{\alpha, \beta = 2}^{N-n} \Big| (\hat A_{ik\alpha} \hat A_{jk\beta} - \hat A_{jk\alpha} \hat A_{ik\beta})\Big|^2.
\end{align}
Using \eqref{normal-curvature-in-principal-direction}, \eqref{decomp-1}, \eqref{decomp-2}, \eqref{decomp-3}, we obtain the following proposition for use in later sections.  
\begin{proposition}[Decomposition of reaction terms]\label{decompreac}
\begin{align}
\label{trace-and-tracless} |h|^2 &= |\mathring h|^2 + \frac{1}{n} |H|^2.\\
\label{A-dot-H} |\langle A_{ij}, H \rangle|^2 & = |H|^2 |h|^2. \\
\label{norm-of-A,A} |\langle A_{ij}, A_{kl} \rangle|^2 & = |h|^4 + 2\Big|\sum_{i,j=1}^n \mathring h_{ij} \hat A_{ij} \Big|^2 + |\langle \hat A_{ij}, \hat A_{kl} \rangle|^2.\\
\label{norm-of-normal-curvature}|R^\perp|^2 &= |\hat R^\perp|^2 + 2|R_{ij}^\perp(\nu_1)|^2,
\end{align}
where
\begin{align}
\label{norm-of-normal-curvature-in-principal-direction}|R_{ij}^\perp(\nu_1)|^2 &= \Big|\sum_{k=1}^n (\mathring h_{ik} \hat A_{jk} - \mathring h_{jk} \hat A_{ik}) \Big|^2,\\
\label{hat-R-perp} |\hat R^\perp|^2 &:= \sum_{\alpha, \beta =2}^{N-n} \Big|\sum_{k = 1}^n (\hat A_{ik\alpha} \hat A_{jk\beta } - \hat A_{jk\alpha} \hat A_{ik\beta})\Big|^2.
\end{align}
\end{proposition}
\noindent The expression $|\hat R^\perp|^2$ will appear in our evolution equation for $|\hat A|^2$. For brevity, from now on we will write
\[
| \mathring h_{ij} \hat A_{ij}|^2 = \Big|\sum_{i,j=1}^n\mathring  h_{ij} \hat A_{ij} \Big|^2.
\]

We now give a proof that the vanishing of $\hat A$ implies codimension one. In application, if the initial flow is $c$-pinched, then the blow up will satisfy $|A|^2 \leq c|H|^2$. The argument below works as long as the tensor $|H| g_{ij} - h_{ij}$ is positive definite (which is equivalent to $h_{ij}$ be $(n-1)$-convex). 
\begin{proposition}\label{whyhatA}
Let $n \geq 2$ and $N > n+1$. Suppose $F : M \to \mathbb{R}^N$ is an immersion of a connected, $n$-dimensional manifold satisfying $|H| > 0$. Assume $|H| g_{ij} - h_{ij}$ is positive definite and $\hat A \equiv 0$. Then $F(M)$ is an immersed hypersurface in an $(n+1)$-dimensional affine subspace of $\mathbb{R}^N$. 
\end{proposition}

\begin{proof}
Because $|H| > 0$, the principal normal $\nu_1$ is well-defined. The vanishing of $\hat A$ in addition to our pinching assumption implies $\nu_1$ is parallel with respect to $\nabla^\perp$. Specifically, recall that by projecting the Codazzi equation, we have \eqref{codazzi-perp}:
\[
|H| \nabla_i^\perp \nu_1 = \sum_{k=1}^n\hat \nabla_k^\perp \hat A_{ik} + h_{ik} \nabla_k^\perp \nu_1 = h_{ik} \nabla_k^\perp \nu_1. 
\]
Since the tensor $|H|g_{ik} - h_{ik}$ is positive definite, we must have $\nabla^\perp \nu_1 = 0$. 

Now let $\gamma : [0, 1] \to M$ be a smooth path connecting any pair of distinct points $p_0 = \gamma(0)$ and $p_1 = \gamma(1)$ in $M$. Define $\nu_{2}(0),\dots, \nu_{N-n}(0)$ to be the completion of $\nu_1(p_0)$ to an orthonormal basis of $N_{p_0}M$. For $\beta \in \{2, \dots, N-n\}$ and $s \in [0, 1]$, let $\nu_\beta(s) \in N_{\gamma(s)}M$ be the parallel transport of $\nu_\beta(0)$ with respect to $\nabla^\perp$. Because $\nu_1$ is parallel with respect to $\nabla^\perp$ and $\langle \nu_\beta(0), \nu_1(p_0) \rangle = 0$, we have $\big \langle \nu_\beta(s), \nu_1\big(\gamma(s)\big) \big\rangle = 0$ for all $s \in [0,1]$. If we let $e_1, \dots, e_n$ denote a parallel orthonormal basis of $T_{\gamma(s)}M$ along $\gamma$, then 
\[
(D_{\gamma'(s)} \nu_\beta(s))^\top =\sum_{i=1}^n \langle D_{\gamma'(s)} \nu_\beta(s), e_i \rangle e_i = - \sum_{i = 1}^n \langle \nu_{\beta}(s), A(\gamma'(s), e_i) \rangle e_i = 0,
\]
since $\nu_{\beta}(s)$ is orthogonal to $A = h \nu_1$. It follows that 
\[
D_{\gamma'(s)} \nu_\beta(s) = \nabla_{\gamma'(s)}^\perp \nu_\beta(s) + (D_{\gamma'(s)} \nu_\beta(s))^\top =0,
\]
which shows $\nu_\beta(s)$ is parallel along $\gamma$ with respect to the ambient connection $D$ as well. On the other hand, the constant unit vector field $\omega_\beta$ in $\mathbb{R}^N$ defined by the condition $\omega_\beta(F(p_0)) = \nu_\beta(0)$, is also parallel along $F(\gamma(s))$ with respect to $D$. By uniqueness of parallel transport, this implies $\nu_\beta(1)$ agrees with the restriction of $\omega_\beta$ to $F(M)$. Since $p_1$ was arbitrary, we see that the restriction of the vector fields $\omega_2, \dots, \omega_{N- n}$ form a parallel orthonormal basis of the complement of $\nu_1$ in $NM$ at every point on $M$. This implies the ambient coordinate functions $y_\beta : \mathbb{R}^{N} \to \mathbb{R}$ given by $y_\beta(x) = \langle x, \omega_{\beta} \rangle$ are constant on $F(M)$.  It follows that $F(M)$ must lie in a translation of the $(n+1)$-dimensional subspace of $\mathbb{R}^N$ orthogonal to $\omega_2, \dots, \omega_{N- n}$. 
\end{proof}


Finally, as the main application of Theorem \ref{main}, we deduce that singularity models must be codimension one. 

\begin{corollary}\label{singularities}
Suppose $n \geq 5$ and $N > n+1 $. Let $c_n =\min\{ \frac{4}{3n}, \frac{3(n+1)}{2n(n+2)}\}$. Suppose $M_t \subset \mathbb{R}^N$, $t \in [0, T)$, is a smooth family of $n$-dimensional closed submanifolds evolving by mean curvature flow which initially satisfies $|A|^2 < c_n |H|^2$. Then at the first singular time every blow-up limit must be codimension one. 
\end{corollary}
\begin{proof}
Since we have strict inequality $|A|^2 < c_n |H|^2$, we can find $\delta > 0$, depending on $M_0$, such that  $f := (c_n-\delta) |H|^2 -|A|^2$ satisfies $f \geq \delta |H|^2$ on $M_0$. Then, by Theorem 2 in \cite{AB10} (note Remark \ref{notation-difference}), $f \geq \delta |H|^2$ for $t \in [0, T)$. Now, suppose $\tilde M_t$ is a smooth blow-up limit. On the blow-up limit, we must also have $\tilde f := (c_n-\delta) |\tilde H|^2 - |\tilde A|^2 \geq \delta |\tilde H|^2 \geq 0$. Since $c_n -\delta < \frac{4}{3n}$, it follows by work in Section 3 of \cite{AB10} that $(\partial_t - \Delta) \tilde f \geq 0$. En route to proving Theorem \ref{main}, we will also establish this inequality in Lemma \ref{evolution-inequality-of-f} below. Hence, by the strong maximum principle, either $\tilde f \equiv 0$ or $\tilde f > 0$ on the blow-up limit. If $\tilde f \equiv 0$, then $\tilde H \equiv 0$, and consequently $\tilde A \equiv 0$. In this case, the blow-up must be a codimension one hyperplane. Otherwise, $\tilde f > 0$, and thus $|\tilde H| > 0$ everywhere on the blow-up. Since $|\tilde A|^2\leq c_n |\tilde H|^2$ easily implies $|\tilde H|\tilde g_{ij} - \tilde h_{ij}$ is positive definite, Proposition \ref{whyhatA} implies the blow-up limit must be codimension one. 
\end{proof}


\section{Evolution of $|\hat A|^2$}

In this section, we compute the evolution equation of $|\hat A|^2$.  We do this by using the formulas stated in Section 2. To begin, we recall the useful standard identity
\begin{equation}\label{useful-identity}
\Big(\frac{\partial}{\partial t} - \Delta\Big)\Big(\frac{u}{v}\Big) =  \frac{1}{v}\Big(\frac{\partial}{\partial t} - \Delta\Big)u - \frac{u}{v^2} \Big(\frac{\partial}{\partial t} - \Delta\Big)v + \frac{2}{v} \nabla_k v \,\nabla_k \Big(\frac{u}{v}\Big).
\end{equation}
Now it follows from \eqref{def-hat-A} that
\[
|\hat A|^2 = |A|^2 - |\langle A_{ij}, H \rangle|^2|H|^{-2}. 
\]
So we will need the evolution equations of $|A|^2$ and $|\langle A_{ij}, H \rangle|^2|H|^{-2}$. We have already recorded, in \eqref{evolution-of-|A|^2} and \eqref{evolution-of-|H|^2}, the evolution equations of $|A|^2$ and $|H|^2$. The latter of these equations combined with \eqref{useful-identity} (set $u = |\langle A_{ij}, H\rangle|^2$ and $v = |H|^2$) implies  
\begin{align}
\label{comp1} \Big(\frac{\partial}{\partial t} - \Delta\Big)\frac{|\langle A_{ij}, H \rangle |^2}{|H|^{2}} &= |H|^{-2} \Big(\frac{\partial}{\partial t} - \Delta \Big)|\langle A_{ij}, H\rangle|^2\\
\nonumber& \qquad  - |H|^{-4}|\langle A_{ij}, H\rangle|^2\big(- 2|\nabla^\perp H|^2 + 2 |\langle A_{kl}, H \rangle|^2\big) \\
\nonumber & \qquad + 2|H|^{-2}\Big\langle \nabla_k |H|^2, \nabla_k \frac{|\langle A_{ij}, H \rangle |^2}{|H|^2} \Big\rangle.
\end{align}
Before computing the evolution of $|\langle A_{ij}, H \rangle|^2$, we simplify the terms on the second and third lines using Propositions \ref{decompgrad} and \ref{decompreac}. In particular, using $|\langle A_{ij}, H \rangle|^2 = |H|^2|h|^2$ and \eqref{nabla-perp-H-squared-formula}, we rewrite the two terms on the second line of \eqref{comp1} as
\begin{align}
\label{comp2} 2|H|^{-4}|\langle A_{ij}, H \rangle|^2|\nabla^\perp H|^2 & = 2|h|^2|\nabla^\perp \nu_1|^2 + 2|H|^{-2}|h|^2|\nabla |H||^2,\\
\label{comp3} -2 |H|^{-4}|\langle A_{ij}, H \rangle|^4 &= - 2|h|^4. 
\end{align}
As for the gradient term on the third line of \eqref{comp1}, we have $\nabla_k |H|^2 = 2|H| \nabla_k|H|$ and  $\nabla_k (|H|^{-2}|\langle A_{ij}, H \rangle|^2) = \nabla_k |h|^2 = 2 h_{ij} \nabla_k h_{ij}$. Therefore, 
\begin{equation}
\label{comp4} 2|H|^{-2}\Big\langle \nabla_k |H|^2, \nabla_k \frac{|\langle A_{ij}, H \rangle |^2}{|H|^2} \Big\rangle  = 8 |H|^{-1} h_{ij} \nabla_k |H| \nabla_k h_{ij}. 
\end{equation}
To summarize \eqref{comp2}, \eqref{comp3}, and \eqref{comp4}, we have shown so far that
\begin{align}
\label{comp5} \Big(\frac{\partial}{\partial t} - \Delta\Big)\frac{|\langle A_{ij}, H \rangle |^2}{|H|^{2}} &= |H|^{-2} \Big(\frac{\partial}{\partial t} - \Delta \Big)|\langle A_{ij}, H\rangle|^2\\
\nonumber & \qquad  - 2|h|^4 + 2|h|^2 |\nabla^\perp \nu_1|^2 + 2|H|^{-2}|h|^2 |\nabla |H||^2 \\
\nonumber & \qquad + 8 |H|^{-1} h_{ij} \nabla_k |H| \nabla_k h_{ij}.
\end{align}
For the evolution of $|\langle A_{ij}, H \rangle|^2$, we have the following lemma. 
\begin{lemma}
\begin{align}
\label{comp6} |H|^{-2} \Big(\frac{\partial}{\partial t} - \Delta \Big)|\langle A_{ij}, H \rangle|^2& = 4|\mathring h_{ij} \hat A_{ij}|^2 + 2|R^\perp_{ij}(\nu_1)|^2 +4|h|^4  \\
\nonumber & \quad -4|H|^{-1} \mathring h_{ij} \nabla_k|H| \langle \nabla_k^\perp \hat A_{ij}, \nu_1 \rangle  -4 \mathring h_{ij} \langle \nabla_k^\perp \hat A_{ij}, \nabla_k^\perp \nu_1 \rangle \\
\nonumber & \quad - 4 |h|^2 |\nabla^\perp\nu_1|^2 -2 |H|^{-2} |h|^2 |\nabla |H||^2 -8|H|^{-1} h_{ij} \nabla_k |H| \nabla_k h_{ij} -2 |\nabla h|^2.
\end{align}
\end{lemma}
\begin{proof}
Recall that any time $h$ is traced with $\hat A$, we may replace $h$ with $\mathring h$ because $\hat A$ is traceless. To begin, we use \eqref{evolution-of-A} and \eqref{evolution-of-H} to obtain 
\begin{align*}
\Big\langle \Big(\frac{\partial}{\partial t}^\perp - \Delta^\perp\Big)A_{ij}, H \Big\rangle & = - \langle H, A_{ik} \rangle \langle A_{jk}, H \rangle - \langle H, A_{jk} \rangle \langle A_{ik}, H \rangle +  \langle A_{ij},  A_{kl} \rangle \langle A_{kl}, H \rangle \\
&\qquad - 2\langle A_{ik}, A_{jl} \rangle \langle A_{kl}, H \rangle + \langle A_{ik}, A_{kl} \rangle \langle A_{jl}, H \rangle + \langle A_{jl},  A_{kl} \rangle \langle A_{ik}, H \rangle, \\
\Big\langle A_{ij},  \Big(\frac{\partial}{\partial t}^\perp - \Delta^\perp\Big)H \Big\rangle &= \langle A_{kl}, H \rangle\langle A_{ij}, A_{kl} \rangle. 
\end{align*}
Tracing each of these equations with a copy of $\langle A_{ij}, H \rangle$, we get
\begin{align*}
\Big\langle \Big(\frac{\partial}{\partial t}^\perp - \Delta^\perp\Big)A_{ij}, H \rangle \langle A_{ij}, H \rangle & = -2 \langle A_{ik}, H \rangle \langle A_{jk}, H \rangle\langle A_{ij}, H \rangle  +  \langle A_{ij},  A_{kl} \rangle \langle A_{kl}, H \rangle \langle A_{ij}, H \rangle\\
&\qquad - 2\langle A_{ik}, A_{jl} \rangle \langle A_{kl}, H \rangle\langle A_{ij}, H \rangle + 2\langle A_{ik}, A_{kl} \rangle \langle A_{jl}, H \rangle \langle A_{ij}, H \rangle, \\
\Big\langle A_{ij},  \Big(\frac{\partial}{\partial t}^\perp - \Delta^\perp\Big)H \Big\rangle \langle A_{ij}, H \rangle&=\langle A_{ij}, A_{kl} \rangle \langle A_{kl}, H \rangle\langle A_{ij}, H \rangle.
\end{align*}
Combining these latter formulas together gives
\begin{align*}
\Big(\Big(\frac{\partial}{\partial t} - \Delta \Big)\langle A_{ij}, H \rangle\Big)\langle A_{ij}, H \rangle & = -2 \langle A_{ik}, H \rangle \langle A_{jk}, H \rangle\langle A_{ij}, H \rangle  +  2\langle A_{ij},  A_{kl} \rangle \langle A_{kl}, H \rangle \langle A_{ij}, H \rangle\\
&\qquad - 2\langle A_{ik}, A_{jl} \rangle \langle A_{kl}, H \rangle\langle A_{ij}, H \rangle + 2\langle A_{ik}, A_{kl} \rangle \langle A_{jl}, H \rangle \langle A_{ij}, H \rangle \\
& \qquad - 2\langle \nabla_k^\perp A_{ij}, \nabla_k^\perp H \rangle \langle A_{ij}, H \rangle.
\end{align*}
Therefore, recalling \eqref{evolution-of-inverse}, we have 
\begin{align}
\nonumber \Big(\frac{\partial}{\partial t} - \Delta \Big)|\langle A_{ij}, H \rangle|^2 &= 2 \Big(\frac{\partial}{\partial t} g^{ij}\Big) g^{kl} \langle A_{ik}, H \rangle \langle A_{jl}, H \rangle + 2\Big(\Big(\frac{\partial}{\partial t} - \Delta \Big)\langle A_{ij}, H \rangle\Big)\langle A_{ij}, H \rangle  - 2|\nabla \langle A_{ij}, H \rangle|^2\\
\label{comp7} & = 4\langle A_{ij},  A_{kl} \rangle \langle A_{kl}, H \rangle \langle A_{ij}, H \rangle\\
\nonumber  &\qquad - 4\langle A_{ik}, A_{jl} \rangle \langle A_{kl}, H \rangle\langle A_{ij}, H \rangle + 4\langle A_{ik}, A_{kl} \rangle \langle A_{jl}, H \rangle \langle A_{ij}, H \rangle \\
\nonumber & \qquad - 4\langle \nabla_k^\perp A_{ij}, \nabla_k^\perp H \rangle \langle A_{ij}, H \rangle - 2|\nabla \langle A_{ij}, H \rangle|^2.
\end{align}
To finish the proof, we multiply by $|H|^{-2}$ and then rewrite each of the remaining terms using $A = \hat A + h\nu$. For the term in the first line of \eqref{comp7}, we have 
\begin{align}
\nonumber 4|H|^{-2}\langle A_{ij},  A_{kl} \rangle \langle A_{kl}, H \rangle \langle A_{ij}, H \rangle &= 4|H|^{-2}|H|^2h_{ij}h_{kl}\langle  A_{ij}, A_{kl} \rangle \\
\nonumber &= 4 |h|^4 + 4h_{ij}h_{kl}\langle  \hat A_{ij}, \hat A_{kl} \rangle \\
\nonumber &= 4 |h|^4 + 4\mathring h_{ij}\mathring h_{kl}\langle  \hat A_{ij}, \hat A_{kl} \rangle \\
\label{comp8}& = 4|h|^4 + 4|\mathring h_{ij} \hat A_{ij} |^2.
\end{align}
For the terms in the second line of \eqref{comp7}, recalling \eqref{norm-of-normal-curvature-in-principal-direction}, we have
\begin{align*}
|R^\perp_{ij}(\nu_1)|^2 &= |\mathring h_{ik} \hat A_{jk} - \mathring h_{jk} \hat A_{ik} |^2 \\
& =  | h_{ik}  A_{jk} -  h_{jk}  A_{ik} |^2 \\
& =\langle  h_{ik}  A_{jk} -  h_{jk}  A_{ik}, h_{il}  A_{jl} -  h_{jl}  A_{il} \rangle \\
& = 2 h_{ik} h_{il} \langle  A_{jk},  A_{jl}\rangle - 2h_{ik} h_{jl} \langle  A_{jk},  A_{il} \rangle \\
& = 2\langle  A_{jk},  A_{jl}\rangle\langle A_{ik}, \frac{H}{|H|} \rangle \langle A_{il}, \frac{H}{|H|} \rangle - 2\langle  A_{jk},  A_{il} \rangle\langle A_{ik}, \frac{H}{|H|} \rangle\langle A_{jl}, \frac{H}{|H|} \rangle.
\end{align*}
After reindexing (e.g. $ j \to k \to l \to i \to j$ on the first term and $j \to i \to l \to j$, $k \to k$ on the second term), this gives
\begin{equation}
\label{comp9} 4|H|^{-2}\big(\langle A_{ik}, A_{kl} \rangle \langle A_{jl}, H \rangle \langle A_{ij}, H \rangle -  \langle A_{ik}, A_{jl} \rangle \langle A_{kl}, H \rangle\langle A_{ij}, H \rangle\big)  = 2 |R^\perp_{ij}(\nu_1)|^2. 
\end{equation}
Together \eqref{comp8} and \eqref{comp9} verify the claimed reaction terms on the first line \eqref{comp6}. 

It remains to analyze terms in third line of \eqref{comp7}. By  \eqref{derivative-of-A} and \eqref{derivative-of-H}
\begin{align*}
\langle \nabla_k^\perp A_{ij}, \nabla_k^\perp H\rangle &= |H| \langle \nabla_k^\perp A_{ij}, \nabla^\perp_k \nu_1 \rangle +\nabla_k|H| \langle \nabla_k^\perp A_{ij}, \nu_1 \rangle \\
& =  |H| \langle \nabla_k^\perp \hat A_{ij}, \nabla^\perp_k \nu_1 \rangle+  |H| h_{ij}|\nabla^\perp \nu_1|^2 +\nabla_k|H| \langle \nabla_k^\perp \hat A_{ij}, \nu_1 \rangle  + \nabla_k|H| \nabla_k h_{ij}. 
\end{align*}
Consequently, the first term on the third line of \eqref{comp7} is
\begin{align}
\nonumber -4|H|^{-2} \langle \nabla_k^\perp A_{ij}, \nabla_k^\perp H \rangle \langle A_{ij}, H \rangle & = -4|H|^{-1}h_{ij}  \langle \nabla_k^\perp A_{ij}, \nabla_k^\perp H \rangle\\
\nonumber& = -4|H|^{-1}  h_{ij}\Big( |H| \langle \nabla_k^\perp \hat A_{ij}, \nabla^\perp_k \nu_1 \rangle+  |H| h_{ij}|\nabla^\perp \nu_1|^2 \\
\nonumber& \qquad \qquad \qquad\qquad  \qquad +\nabla_k|H| \langle \nabla_k^\perp \hat A_{ij}, \nu_1 \rangle  + \nabla_k|H| \nabla_k h_{ij}\Big)\\
\label{comp10}& = -4|H|^{-1}\mathring h_{ij} \nabla_k|H| \langle \nabla_k^\perp \hat A_{ij}, \nu_1 \rangle  -4|H|^{-1} h_{ij} \nabla_k|H| \nabla_k h_{ij}\\
\nonumber & \qquad   -4 \mathring h_{ij} \langle \nabla_k^\perp \hat A_{ij}, \nabla_k^\perp \nu_1 \rangle - 4 |h|^2 |\nabla^\perp \nu_1|^2. 
\end{align}
As for the second term on the third line of \eqref{comp7}, we have
\begin{align}
\nonumber -2|H|^{-2} |\nabla \langle A_{ij}, H \rangle |^2 &= -2|H|^{-2}|\nabla (|H|h_{ij}) |^2 \\
\label{comp11}&= -2|H|^{-2} |h|^2 |\nabla |H||^2 -2 |\nabla h|^2 - 4|H|^{-1} h_{ij} \nabla_k |H| \nabla_k h_{ij}. 
\end{align}
Equations \eqref{comp10} and \eqref{comp11} give the claimed six gradient terms on the right-hand side of equation \eqref{comp6}. This completes the proof of the lemma.
\end{proof}

Substituting the lemma statement \eqref{comp6} into \eqref{comp5} and combining like terms yields 
\begin{align*}
\Big(\frac{\partial}{\partial t} - \Delta\Big)\frac{|\langle A_{ij}, H \rangle |^2}{|H|^{2}} &= 4|\mathring h_{ij}\hat A_{ij}|^2 + 2|R^\perp_{ij}(\nu_1)|^2 + 2|h|^4 \\
& \qquad -4|H|^{-1}\mathring  h_{ij} \nabla_k|H| \langle \nabla_k^\perp \hat A_{ij}, \nu_1 \rangle  -4 \mathring h_{ij} \langle \nabla_k^\perp \hat A_{ij}, \nabla_k^\perp \nu_1 \rangle\\
& \qquad  -2 |\nabla h|^2 - 2|h|^2 |\nabla^\perp \nu_1|^2.
\end{align*}
We negate the expression above and add \eqref{evolution-of-|A|^2}, the evolution equation of $|A|^2$, to get
\begin{align}
\label{comp12} \Big(\frac{\partial}{\partial t} - \Delta\Big) |\hat A|^2 & = - 2|\nabla^\perp A|^2 +  2|\langle A_{ij}, A_{kl} \rangle|^2 + 2|R^\perp|^2\\
\nonumber & \qquad  - 4|\mathring h_{ij} \hat A_{ij}|^2 - 2|R^\perp_{ij}(\nu_1)|^2 - 2|h|^4 \\
\nonumber & \qquad + 4|H|^{-1} \mathring h_{ij} \nabla_k|H| \langle \nabla_k^\perp \hat A_{ij}, \nu_1 \rangle  + 4 \mathring h_{ij} \langle \nabla_k^\perp \hat A_{ij}, \nabla_k^\perp \nu_1 \rangle\\
\nonumber & \qquad  + 2 |\nabla h|^2 + 2|h|^2 |\nabla^\perp \nu_1|^2.
\end{align}
By \eqref{norm-of-A,A} and \eqref{norm-of-normal-curvature} in Proposition \ref{decompreac}, the reaction terms above satisfy
\begin{align}
\label{comp13}2 |\langle A_{ij}, A_{kl} \rangle|^2  - 4|\mathring h_{ij} \hat A_{ij}|^2  - 2|h|^4 &= 2|\langle \hat A_{ij},\hat A_{kl} \rangle|^2, \\
\label{comp14}2|R^\perp_{ij}|^2 - 2|R_{ij}^\perp(\nu_1)|^2 &= 2|\hat R^\perp|^2 + 2|R^\perp_{ij}(\nu_1)|^2.
\end{align}
As for the gradient terms, using \eqref{derivative-of-A}, we obtain  
\begin{align*}
|\nabla^\perp A|^2 = |\nabla^\perp \hat A|^2 + |\nabla h|^2 + |h|^2 |\nabla^\perp \nu_1|^2 + 2 \mathring h_{ij} \langle \nabla^\perp \hat A_{ij}, \nabla_k^\perp \nu_1 \rangle + 2 \nabla_k \mathring h_{ij} \langle \nabla_k^\perp \hat A_{ij}, \nu_1 \rangle. 
\end{align*} 
Rearranging this gives 
\begin{align}
\label{comp15}- 2|\nabla^\perp A|^2 + 2 |\nabla h|^2+ 2|h|^2 |\nabla^\perp \nu_1|^2 + 4 \mathring h_{ij} \langle \nabla_k^\perp \hat A_{ij}, \nabla_k^\perp \nu_1 \rangle =  -2|\nabla^\perp \hat A|^2 - 4\nabla_k \mathring h_{ij} \langle \nabla_k^\perp \hat A_{ij}, \nu_1\rangle. 
\end{align}
Substituting \eqref{comp13}, \eqref{comp14}, and \eqref{comp15} into equation \eqref{comp12}, we finally get
\begin{align}
\label{comp16}\Big(\frac{\partial}{\partial t} - \Delta\Big) |\hat A|^2 & = 2|\langle\hat  A_{ij}, \hat A_{kl} \rangle|^2 +  2|\hat R^\perp |^2 + 2|R^\perp_{ij}(\nu_1)|^2\\
\nonumber& \qquad - 2|\nabla^\perp \hat A|^2 + 4|H|^{-1} \mathring h_{ij} \nabla_k|H| \langle \nabla_k^\perp \hat A_{ij}, \nu_1 \rangle  - 4 \nabla_k \mathring h_{ij} \langle \nabla_k^\perp \hat A_{ij}, \nu_1 \rangle.
\end{align}
Note that, by the orthogonality relations \eqref{orthgonality-relations} and the identity $\mathring A = \hat A + \mathring h \nu_1$, 
\begin{align*}
\langle \nabla_k^\perp \hat A_{ij}, \nu_1 \rangle &= -\langle \hat A_{ij}, \nabla_k^\perp \nu_1 \rangle = -\langle \mathring A_{ij},\nabla_k^\perp \nu_1 \rangle,\\ 
\nabla_k \mathring h_{ij} &= \langle \nabla_k^\perp \mathring A_{ij}, \nu_1 \rangle - \langle \nabla_k^\perp \hat A_{ij}, \nu_1 \rangle. 
\end{align*}
From these and \eqref{comp16}, we may now conclude the following proposition. 
\begin{proposition}[Evolution of $|\hat A|^2$]\label{evolution-of-hatA}
\begin{align}
\label{evolution-of-|hatA|^2} \Big(\frac{\partial}{\partial t}  - \Delta\Big)|\hat A|^2 & =  2 | \langle \hat A_{ij}, \hat A_{kl} \rangle |^2 + 2 |\hat R^\perp|^2 + 2|R^\perp_{ij}(\nu_1)|^2 \\
\nonumber & \qquad - 2|\nabla^\perp \hat A|^2 + 4 \sum_{i,j,k = 1}^nQ_{ijk} \langle  \hat A_{ij}, \nabla^\perp_k\nu_1\rangle,
\end{align}
where 
\begin{equation}
\label{definition-of-Q} Q_{ijk} := \langle \nabla_k^\perp \mathring A_{ij}, \nu_1 \rangle - \langle \nabla_k^\perp \hat A_{ij}, \nu_1 \rangle - |H|^{-1} \mathring h_{ij} \nabla_k|H|.
\end{equation}
\end{proposition}


\section{Proof of Theorem 1.1.}

In a similar fashion to the proof of the pointwise derivative estimate in \cite{Ngu18}, will compare the evolution of $|\hat A|^2$ to the evolution of $c |H|^2 -|A|^2$ along the mean curvature flow. Recall the setup of Theorem \ref{main}: we are given an $n$-dimensional mean curvature flow $M_t = F(M,t) \subset \mathbb{R}^N$ which satisfies $c_n |H|^2 > |A|^2$ initially, where
\[
c_n := \begin{cases} \frac{3(n+1)}{2n(n+2)} & \text{ if } \; n = 5, 6, \text{ or } 7 \\ \frac{4}{3n} & \text{ if } \; n \geq 8 \end{cases}. 
\]
Note $c_n = \min\{\frac{3(n+1)}{2n(n+2)}, \frac{4}{3n} \}$. As $M$ is compact, we can find small real numbers $\varepsilon_0, \varepsilon_1>0$ depending on $M_0$ such that $(c_n - \varepsilon_0) |H|^2 - |A|^2 > \varepsilon_1$ holds everywhere on $M$ initially. Define a new constant
\[
c_0 :=  \begin{cases} \frac{3(n+1)}{2n(n+2)} -\varepsilon_0& \text{ if } \; n = 5, 6, \text{ or } 7 \\ \frac{4}{3n} & \text{ if } \; n \geq 8 \end{cases}
\]
and 
\begin{equation}
\label{definition-of-f} f:= c_0 |H|^2 - |A|^2. 
\end{equation}
Note that $c_0 \leq c_n$ only depends upon initial data if $n =5, 6$, or $7$. The subtraction of $\varepsilon_0$ in these lower dimensions is because we need a bit more breathing room in our estimates than the critical dimensional constant $\frac{3(n+1)}{2n(n+2)}$ allows. By Theorem 2 in \cite{AB10} (note Remark \ref{notation-difference}), we have $f > \varepsilon_1$ on $M \times [0, T)$. In particular, $|H| > 0$ on $M \times [0, T)$.  

Let $\delta > 0$ be a small constant, which we will determine towards the end of the proof. We computed the evolution equation of $|\hat A|^2$ in the previous section. By \eqref{evolution-of-|A|^2} and \eqref{evolution-of-|H|^2} in Proposition \ref{evolution-equations}, the evolution equation for $f$ is given by
\begin{equation}
\label{evolution-equation-of-f}\Big(\frac{\partial}{\partial t}- \Delta \Big)f = 2(|\nabla^\perp A|^2 - c_0 |\nabla^\perp H|^2)+ 2(c_0|\langle A_{ij}, H \rangle|^2 - |\langle A_{ij}, A_{kl} \rangle|^2 - |R_{ij}^\perp|^2). 
\end{equation}
The pinching condition implies that both terms on the right-hand side of \eqref{evolution-equation-of-f} are nonnegative on $M \times [0, T)$. This is proven in \cite{AB10} and stated in \cite{Ngu18}, but we will establish this ourselves in Lemma \ref{evolution-inequality-of-f} below.

The first step of the proof of Theorem \ref{main}, and our main task, is to analyze the evolution equation of the scale-invariant quantity $|\hat A|^2/f$. We will show this ratio satisfies a favorable evolution equation with a right-hand side that has a nonpositive term. Specifically, we will show that 
\begin{equation}
\label{goal} \Big(\frac{\partial}{\partial t} - \Delta\Big)\frac{|\hat A|^2}{f} \leq 2 \Big\langle \nabla \frac{|\hat A|^2}{f} , \nabla \log f \Big\rangle - \delta\frac{|\hat A|^2}{f^2} \Big(\frac{\partial}{\partial t} - \Delta \Big) f.
\end{equation}
Then we will analyze the evolution of the non-scale-invariant quantity $|\hat A|^2/f^{1-\sigma}$. We will show for $\sigma > 0$ sufficiently small, the nonpositive term above can be used to control the nonnegative terms introduced by the additional factor of $f^{\sigma}$. Theorem \ref{main} will then follow from the maximum principle. 

By equation \eqref{evolution-equation-of-f}, Proposition \ref{evolution-of-hatA}, and the useful identity \eqref{useful-identity}, the evolution equation of $|\hat A|^2/f$ is given by
\begin{align*}
\Big(\frac{\partial}{\partial t} - \Delta\Big)\frac{|\hat A|^2}{f} & = \frac{1}{f}\Big(\frac{\partial}{\partial t} - \Delta\Big)|\hat A|^2 - |\hat A|^2\frac{1}{ f^2}\Big(\frac{\partial}{\partial t} -\Delta\Big)f + 2 \Big\langle \nabla \frac{|\hat A|^2}{f}, \nabla \log  f \Big\rangle\\
& = \frac{1}{ f} \Big( 2|\langle \hat A_{ij}, \hat A_{kl} \rangle|^2 + 2|\hat R^\perp|^2 + 2|R^\perp_{ij}(\nu_1)|^2\Big) + \frac{1}{f} \Big(- 2|\nabla^\perp \hat A|^2 + 4 Q_{ijk} \langle \hat A_{ij}, \nabla_k^\perp \nu_1 \rangle\Big) \\
& \quad - |\hat A|^2\frac{1}{ f^2} \Big(2(|\nabla^\perp A|^2-c_0|\nabla^\perp H|^2)\Big)  - |\hat A|^2\frac{1}{ f^2} \Big( 2(c_0|\langle A_{ij}, H \rangle|^2 - |\langle A_{ij}, A_{kl} \rangle|^2 - |R^\perp_{ij}|^2)\Big)\\
 & \quad+  2 \Big\langle \nabla \frac{|\hat A|^2}{ f}, \nabla \log f \Big\rangle.
 \end{align*}
Rearranging these terms, we have
 \begin{align}
\label{equation-of-ratio}\Big(\frac{\partial}{\partial t} - \Delta\Big)\frac{|\hat A|^2}{f}  & = \frac{1}{ f} \Big(2|\langle \hat A_{ij}, \hat A_{kl} \rangle|^2 + 2|\hat R^\perp|^2 + 2|R^\perp_{ij}(\nu_1)|^2 \Big)\\
\nonumber & \quad - \frac{1}{f}\Big(2\frac{|\hat A|^2}{ f} (c_0|\langle A_{ij}, H \rangle|^2 - |\langle A_{ij}, A_{kl} \rangle|^2 - |R^\perp_{ij}|^2)\Big)\\
\nonumber & \quad + \frac{1}{ f} \Big(4Q_{ijk} \langle \hat A_{ij}, \nabla^\perp_k \nu_1 \rangle -  2|\nabla^\perp \hat A|^2 - 2\frac{|\hat A|^2}{ f}(|\nabla^\perp A|^2-c_0|\nabla^\perp H|^2)\Big) \\
\nonumber & \quad+  2 \Big\langle \nabla \frac{|\hat A|^2}{ f}, \nabla \log f \Big\rangle.
\end{align}

We analyze the right-hand side in two steps. We must estimate the reaction terms on the first line by the reaction terms on the second line and the gradient term $4Q_{ijk}\langle \hat A_{ij}, \nabla_k^\perp \nu_1 \rangle$ by the good Bochner terms coming from the evolution of $|\hat A|^2$ and $f$.  \\

\subsection{Reaction Term Estimates}
We begin by estimating the reaction terms. We will make use of the following two estimates. The first estimate is proven on page 372 in \cite{AB10} Section 3. The second estimate is a matrix inequality which is Theorem 1 in \cite{LL92}. 
\begin{lemma}\label{ab_est}
\begin{align}
\label{simpler-inequality}\Big|\sum_{i,j =1}^n \mathring h_{ij} \hat A_{ij} \Big|^2 + |R_{ij}^\perp(\nu_1)|^2 &\leq 2|\mathring h|^2|\hat A|^2, \\
\label{matrix-inequality} |\langle \hat A_{ij}, \hat A_{kl} \rangle|^2 + |\hat R^\perp|^2 &\leq \frac{3}{2}|\hat A|^4.
\end{align}
\end{lemma}
\begin{proof}
The arguments given in \cite{AB10} to prove inequality \eqref{simpler-inequality} are simple and short, so we will repeat them in our notation here. We will express inequality \eqref{matrix-inequality} so that it is an immediate consequence of Theorem 1 in \cite{LL92}. 

Fix any point $p \in M$ and time $t \in [0, T)$. Let $e_1, \dots, e_n$ be an orthonormal basis which identifies $T_pM \cong \mathbb{R}^n$ at time $t$ and then choose $\nu_2, \dots, \nu_{N-n}$ to be a basis of the orthogonal complement of principal normal $\nu_1$ in $N_pM$ at time $t$. For each $\beta \in \{2, \dots, N-n\}$, define a matrix $A_{\beta} = \langle A, \nu_\beta\rangle$ whose components are given by $(A_\beta)_{ij} = A_{ij\beta}$.  Note, by definition, $\hat A_{ij\beta} = A_{ij\beta}$ when $\beta \geq 2$, but we want to match the notation of \cite{LL92} for the moment. Then $\hat A = \sum_{\beta =2}^{N-n} A_\beta \nu_\beta$. We also have $\mathring h = \langle \mathring A, \nu_1\rangle$. 

To prove \eqref{simpler-inequality}, let $\lambda_1, \dots, \lambda_n$ denote the eigenvalues of $\mathring h$. Assume the orthonormal basis is an eigenbasis of $\mathring h$.  Now
\begin{align*}
\Big|\sum_{i,j =1}^n \mathring h_{ij} \hat A_{ij} \Big|^2 &= \sum_{\beta = 2}^{N-n} \sum_{i,j,k,l =1}^n \mathring h_{ij} \mathring h_{kl} A_{ij\beta} A_{kl\beta} \\
&= \sum_{\beta =2}^{N-n}\Big(\sum_{i,j =1}^n \mathring h_{ij} A_{ij\beta}\Big)^2 \\
& = \sum_{\beta =2}^{N-n}\Big(\sum_{i=1}^n \lambda_iA_{ii\beta}\Big)^2.
\end{align*}
By Cauchy-Schwarz, 
\begin{equation}
\label{ineq1}\Big|\sum_{i,j =1}^n \mathring h_{ij} \hat A_{ij} \Big|^2 \leq \sum_{\beta =2}^{N-n}\Big(\sum_{i=1}^n \lambda_j^2\Big) \Big(\sum_{i=1}^n A_{ii\beta}^2\Big) =|\mathring h|^2 \sum_{\beta =2}^{N-n} \sum_{i=1}^n A_{ii\beta}^2.
\end{equation}
By \eqref{norm-of-normal-curvature-in-principal-direction},
\begin{align*}
|R^\perp_{ij}(\nu_1)|^2 &=\sum_{\beta = 2}^{N-2}  \sum_{i,j =1}^n \Big(\sum_{k=1}^n (\mathring h_{ik} A_{jk\beta} - \mathring h_{jk} A_{ik\beta})\Big)^2 \\
& = \sum_{\beta = 2}^{N-2} \sum_{i,j =1}^n (\lambda_i - \lambda_j)^2 A_{ij \beta}^2 \\
& =  \sum_{\beta = 2}^{N-2}\sum_{i \neq j} (\lambda_i - \lambda_j)^2 A_{ij \beta}^2. 
\end{align*}
Since $(\lambda_i - \lambda_j)^2 \leq 2(\lambda_i^2 + \lambda_j^2) \leq 2|\mathring h|^2$, we have
\begin{equation}
\label{ineq2}|R^\perp_{ij}(\nu_1)|^2 \leq 2|\mathring h|^2 \sum_{\beta = 2}^{N-2}  \sum_{i \neq j}A_{ij \beta}^2.
\end{equation}
Summing \eqref{ineq1} and \eqref{ineq2}, we obtain 
\[
\Big|\sum_{i,j =1}^n \mathring h_{ij} \hat A_{ij} \Big|^2 + |R^\perp_{ij}(\nu_1)|^2 \leq |\mathring h|^2 \sum_{\beta =2}^{N-n} \sum_{i=1}^n A_{ii\beta}^2  + 2|\mathring h|^2 \sum_{\beta = 2}^{N-2} \sum_{i \neq j}A_{ij \beta}^2 \leq 2 |\mathring h|^2 |\hat A|^2, 
\]
which is \eqref{simpler-inequality}

To establish \eqref{matrix-inequality}, for $\alpha,\beta \in \{2, \dots, N-n\}$ define
\[
S_{\alpha\beta} := \mathrm{tr}( A_\alpha  A_\beta) = \sum_{i, j =1}^n A_{ij\alpha}A_{ij\beta}\qquad \text{and} \qquad S_{\alpha} := |A_\alpha|^2 = \sum_{i, j =1}^n A_{ij\alpha}A_{ij\alpha}. 
\]
Let $S := S_2 + \cdots + S_{N-n} = |\hat A|^2$. 

Now
\begin{align*}
 |\langle \hat A_{ij}, \hat A_{kl} \rangle|^2 &= \sum_{i,j,k,l =1}^n \sum_{\alpha, \beta =2}^{N-n} A_{ij\alpha}A_{kl\alpha}A_{ij\beta}A_{kl\beta}\\
 & = \sum_{\alpha\beta = 2}^{N-n} \Big(\sum_{i, j = 1}^n A_{ij\alpha}A_{ij\beta}\Big)\Big(\sum_{k,l= 1}^n A_{kl\alpha}A_{kl\beta}\Big) \\
 & = \sum_{\alpha,\beta =2}^{N-n} S_{\alpha\beta}^2. 
\end{align*}
In addition, recalling the definition \eqref{hat-R-perp}, we may write 
\[
|\hat R^\perp|^2 = \sum_{\alpha,\beta =2}^{N-n} \big|A_\alpha A_\beta - A_\beta A_\alpha\big|^2,
\]
where $(A_\alpha A_\beta)_{ij} = (A_\alpha)_{ik} (A_\beta)_{kj} =(A_\alpha)_{ik} (A_\beta)_{jk}$ denotes standard matrix multiplication and $|\cdot |$ is the usual square norm of the matrix. We see that inequality \eqref{matrix-inequality} is equivalent to 
\begin{equation}
\label{ll92}\sum_{\alpha,\beta =2}^{N-n} \big|A_\alpha A_\beta - A_\beta A_\alpha\big|^2 +  \sum_{\alpha,\beta =2}^{N-n} S_{\alpha\beta}^2 \leq \frac{3}{2} S^2. 
\end{equation}
Now if $N- n =2$, inequality \eqref{matrix-inequality} is trivial since $|R^\perp|^2 = 0$ and $|\langle \hat A_{ij}, \hat A_{kl} \rangle|^2 = |\hat A|^4$. Otherwise, if $N - n \geq 3$, inequality \eqref{ll92} is precisely Theorem 1 in \cite{LL92}. This completes the proof. 
\end{proof}

As an immediate consequence of the previous lemma, we have the following estimate for the reaction terms coming from the evolution of $|\hat A|^2$. 
\begin{lemma}[Upper bound for the reaction terms of $(\partial_t - \Delta)|\hat A|^2$] \label{upperA}
\begin{equation}
\label{reaction-term-estimate-hat-A}|\langle \hat A_{ij}, \hat A_{kl} \rangle|^2 +  |\hat R^\perp|^2 + |R^\perp_{ij}(\nu_1)|^2 \leq \frac{3}{2}|\hat A|^4 + 2|\mathring h|^2|\hat A|^2.
\end{equation}
\end{lemma}

Next we express the reaction term in the evolution of $f$ in terms of $\hat A$, $\mathring h$, and $|H|$. In view of \eqref{definition-of-f}, observe that
\begin{equation}
\label{H-terms}\frac{nc_0 -1}{n}|H|^2 = |\hat A|^2 + |\mathring h|^2 + f. 
\end{equation}
We have following lower bound for the reaction terms in the evolution of $f$. 
\begin{lemma}[Lower bound for the reaction terms of $(\partial_t - \Delta)f$]\label{lowerf}
If $\frac{1}{n} < c_0 \leq \frac{4}{3n}$, then
\begin{equation}
\label{lower-bound-reaction-terms-of-f}\frac{|\hat A|^2}{f} (c_0|\langle A_{ij}, H \rangle|^2 - |\langle A_{ij}, A_{kl} \rangle|^2 - |R^\perp_{ij}|^2) \geq \frac{2}{nc_0 -1} |\hat A|^4 +  \frac{nc_0}{nc_0- 1}|\mathring h|^2 |\hat A|^2.
\end{equation}
\end{lemma}
\begin{proof}
We do a computation that is similar to a computation in \cite{AB10}, except we do not throw away the pinching term $f$. By \eqref{trace-and-tracless}, \eqref{A-dot-H}, \eqref{norm-of-A,A}, \eqref{norm-of-normal-curvature}, we have
\begin{align*}
&c_0|\langle A_{ij}, H\rangle |^2 - |\langle A_{ij}, A_{kl} \rangle|^2 - |R_{ij}^\perp|^2 \\
& \quad= \frac{1}{n}c_0|H|^4 + c_0|\mathring h|^2 |H|^2 \\
& \quad \quad - |\mathring h|^4 - \frac{2}{n}|\mathring h|^2 |H|^2 - \frac{1}{n^2}|H|^4 - 2|\mathring h^{ij}\hat A_{ij}|^2 - |\langle \hat A_{ij}, \hat A_{kl} \rangle|^2 \\
& \quad \quad - |\hat R^\perp|^2 - 2 |R_{ij}^\perp(\nu_1)|^2 \\
&\quad =\frac{1}{n}(c_0 - \frac{1}{n})|H|^4+ (c_0 - \frac{1}{n} )|\mathring h|^2 |H|^2  - \frac{1}{n}|\mathring h|^2 |H|^2-|\mathring h|^4 \\
& \quad \quad   - 2|\mathring h_{ij} \hat A_{ij} |^2 - 2|R_{ij}^\perp(\nu_1)|^2 - |\langle \hat A_{ij}, \hat A_{kl} \rangle|^2  - |\hat R^\perp|^2.
\end{align*}
Use \eqref{H-terms} and cancel terms to get
\begin{align*}
&c_0|\langle A_{ij}, H\rangle |^2 - |\langle A_{ij}, A_{kl} \rangle|^2 - |R_{ij}^\perp|^2 \\
& \quad= \frac{1}{n}(f  + |\hat A|^2 + |\mathring h|^2)|H|^2 + (f + |\hat A|^2 + |\mathring h|^2 )|\mathring h|^2 - \frac{1}{n}|\mathring h|^2 |H|^2  -|\mathring h|^4 \\
&\quad \quad  - 2|\mathring h_{ij} \hat A_{ij} |^2 - 2|R_{ij}^\perp(\nu_1)|^2 - |\langle \hat A_{ij}, \hat A_{kl} \rangle|^2  - |\hat R^\perp|^2 \\
&\quad = \frac{1}{n}(f  + |\hat A|^2)|H|^2 + (f + |\hat A|^2 )|\mathring h|^2 \\
&\quad\quad  - 2|\mathring h_{ij} \hat A_{ij} |^2 - 2|R_{ij}^\perp(\nu_1)|^2 - |\langle \hat A_{ij}, \hat A_{kl} \rangle|^2 -|\hat R^\perp|^2.
\end{align*}
Using \eqref{H-terms} once more for the remaining factor of $|H|^2$ gives
\begin{align*}
&c_0|\langle A_{ij}, H\rangle |^2 - |\langle A_{ij}, A_{kl} \rangle|^2 - |R_{ij}^\perp|^2 \\
& \quad= \frac{1}{n}(f  + |\hat A|^2)\Big(c_0 -\frac{1}{n}\Big)^{-1}(f + |\hat A|^2 + |\mathring h|^2) + (f + |\hat A|^2 )|\mathring h|^2 \\
&\quad\quad  - 2|\mathring h_{ij} \hat A_{ij} |^2 - 2|R_{ij}^\perp(\nu_1)|^2 - |\langle \hat A_{ij}, \hat A_{kl} \rangle|^2  -|\hat R^\perp|^2 \\
& \quad= \frac{1}{nc_0 -1}f(f +2|\hat A|^2+|\mathring h|^2) + f|\mathring h|^2+ \frac{1}{nc_0 -1} |\hat A|^4 + \frac{nc_0}{nc_0 -1}|\hat A|^2|\mathring h|^2 \\
&\quad\quad  - 2|\mathring h_{ij} \hat A_{ij} |^2 - 2|R_{ij}^\perp(\nu_1)|^2 - |\langle \hat A_{ij}, \hat A_{kl} \rangle|^2  - |\hat R^\perp|^2. 
\end{align*}
Now by the two estimates in Lemma \ref{ab_est}
\[
2|\mathring h_{ij} \hat A_{ij} |^2 + 2|R_{ij}^\perp(\nu_1)|^2 +  |\langle \hat A_{ij}, \hat A_{kl} \rangle|^2  +|\hat R^\perp|^2 \leq 4|\mathring h|^2 |\hat A|^2 + \frac{3}{2}|\hat A|^4. 
\]
Therefore
\begin{align*}
\frac{1}{nc_0 -1} &|\hat A|^4 + \frac{nc_0}{nc_0 -1}|\hat A|^2|\mathring h|^2 - 2|\mathring h^{ij} \hat A_{ij} |^2 - 2|R_{ij}^\perp(\nu_1)|^2 - |\langle \hat A_{ij}, \hat A_{kl} \rangle|^2  - |\hat R^\perp|^2 \\
& \geq \Big(\frac{1}{nc_0 -1} - \frac{3}{2} \Big)|\hat A|^4 + \Big(\frac{nc_0}{nc_0 -1} - 4 \Big)|\mathring h|^2|\hat A|^2. 
\end{align*}
Since $c_0 \leq \frac{4}{3n}$, we have
\[
\frac{1}{nc_0 -1} - \frac{3}{2} \geq \frac{3}{2}, \qquad \frac{nc_0}{nc_0 -1} - 4 \geq 0. 
\]
Consequently, we have
\begin{align}
\label{for-evolution-of-f} c_0|\langle A_{ij}, H\rangle |^2 - |\langle A_{ij}, A_{kl} \rangle|^2 - |R_{ij}^\perp|^2  &\geq \frac{2}{nc_0 -1} f |\hat A|^2 + \frac{nc_0}{nc_0 -1} f|\mathring h|^2 + \frac{1}{nc_0-1} f^2\\
\nonumber & \geq \frac{2}{nc_0 -1} f |\hat A|^2 + \frac{nc_0}{nc_0 -1} f|\mathring h|^2.
\end{align}
Multiplying both sides by $\frac{|\hat A|^2}{f}$ completes the proof of the lemma. 
\end{proof}

Putting Lemmas \ref{upperA} and \ref{lowerf} together, we have 
\begin{lemma}[Reaction term estimate]\label{reactionest}
If $0 < \delta \leq \frac{1}{2}$ and $\frac{1}{n} < c_0 \leq \frac{4}{3n}$, then
\begin{equation}
\label{reaction-term-estimate-final}|\langle \hat A_{ij}, \hat A_{kl} \rangle|^2 +  |\hat R^\perp|^2  + |R^\perp_{ij}(\nu_1)|^2  \leq (1-\delta) \frac{|\hat A|^2}{f} (c_0|\langle A_{ij}, H \rangle|^2 - |\langle A_{ij}, A_{kl} \rangle|^2 - |R^\perp_{ij}|^2). 
\end{equation}
\end{lemma}
\begin{proof}
In view of \eqref{reaction-term-estimate-hat-A} and \eqref{lower-bound-reaction-terms-of-f}, we have
\begin{align*}
|\langle \hat A_{ij}, \hat A_{kl} \rangle|^2 +  |\hat R^\perp|^2&  + |R^\perp_{ij}(\nu_1)|^2 - (1-\delta) \frac{|\hat A|^2}{f} (c_0|\langle A_{ij}, H \rangle|^2 - |\langle A_{ij}, A_{kl} \rangle|^2 - |R^\perp_{ij}|^2) \\
& \leq \frac{3}{2}|\hat A|^4 + 2|\mathring h|^2|\hat A|^2 -  \frac{2(1-\delta)}{nc_0 -1} |\hat A|^4- \frac{nc_0(1-\delta)}{nc_0 -1}|\mathring h|^2 |\hat A|^2 \\
& = \Big(\frac{3}{2} - \frac{2(1-\delta)}{nc_0 -1}\Big)|\hat A|^4 + \Big(2 - \frac{nc_0(1-\delta)}{nc_0 -1}\Big)|\mathring h|^2|\hat A|^2. 
\end{align*}
If $c_0 \leq \frac{4}{3n}$, then
\[
\frac{1}{nc_0 -1} \geq 3,\quad \text{and} \quad \frac{nc_0}{nc_0 -1} \geq 4. 
\]
Therefore if $\delta \leq \frac{1}{2}$ 
\begin{align*}
\frac{3}{2}  - \frac{2(1-\delta)}{nc_0 -1} &\leq \frac{3}{2} - 6(1-\delta) \leq 0,\\
2- \frac{nc_0(1-\delta)}{nc_0 -1} &\leq 2 -  4(1-\delta) \leq 0,
\end{align*}
which gives \eqref{reaction-term-estimate-final}.
\end{proof}

\subsection{Gradient Term Estimates}
Having analyzed the reaction terms, we turn our attention to the gradient terms. For this, we will use equation \eqref{norm-of-derivative-of-A}. Recalling that $\hat A_{jk}$ is traceless, it is straightforward to verify that
\begin{align}
\label{term-1-derivative-of-A}|\nabla_i h_{jk} + \langle \nabla_i^\perp \hat A_{jk}, \nu_1 \rangle|^2 & = |\nabla_i \mathring h_{jk} + \langle \nabla_i^\perp \hat A_{jk}, \nu_1 \rangle|^2 + \frac{1}{n} |\nabla |H||^2,\\
\label{term-perp-derivative-of-A}|\hat \nabla_i^\perp \hat A_{jk} +  h_{jk} \nabla_i^\perp \nu_1  |^2 &= |\hat \nabla_i^\perp \hat A_{jk} + \mathring h_{jk} \nabla_i^\perp \nu_1  |^2 + \frac{1}{n} |H|^2 |\nabla^\perp \nu_1|^2.
\end{align}
Observe that the first term in \eqref{term-1-derivative-of-A} is just
\begin{equation}
\label{traceless-derivative-term} |\langle \nabla_i^\perp \mathring A_{jk}, \nu_1 \rangle|^2 =  |\nabla_i \mathring h_{jk} + \langle \nabla_i^\perp \hat A_{jk}, \nu_1 \rangle|^2,
\end{equation}
which will be useful later on. 
Now as observed in \cite{Hui84} (cf. \cite{Ham82} for the Ricci flow), it follows from the Codazzi identity \eqref{codazzi} that the tensor
\[
E_{ijk} = \frac{1}{n+2}\Big(g_{ij} \nabla_k^\perp H + g_{jk} \nabla_i^\perp H + g_{ki} \nabla_j^\perp H\Big)
\]
is an irreducible component of $\nabla_i^\perp A_{jk}$ consisting of its various traces. In other words, $E_{ijk} \nabla_i^\perp A_{jk} = |E|^2$. This allows one to get an improved estimate over the trivial one. Namely, 
\begin{equation}
\label{classic-improved-derivative-estimate} |E|^2 = \frac{3}{n+2}|\nabla^\perp H|^2 \leq |\nabla^\perp A|^2.
\end{equation}
By consequence of this estimate and \eqref{for-evolution-of-f}, we can conclude nonnegativity of $(\partial_t - \Delta)f$: 

\begin{lemma}[Nonnegativity of $(\partial_t -\Delta)f$]\label{evolution-inequality-of-f}
If $\frac{1}{n} < c_0 \leq \frac{4}{3n}$, then 
\begin{align*}
\Big(\frac{\partial}{\partial t} - \Delta\Big)f &\geq 2 |h|^2 f +  \frac{2}{nc_0 -1}|\hat A|^2 f + \Big(1 - \frac{c_0(n+2)}{3}\Big) |\nabla^\perp A|^2.
\end{align*}
In particular, $(\partial_t - \Delta)f \geq 0$ if $f\geq 0$. 
\end{lemma}
\begin{proof}
We first observe, using \eqref{H-terms}, that 
\begin{align*}
 \frac{2}{nc_0 -1} f |\hat A|^2 + \frac{nc_0}{nc_0 -1} f|\mathring h|^2 + \frac{1}{nc_0-1} f^2 &=\frac{1}{nc_0 -1}(|\mathring h|^2 + |\hat A|^2 + f)f + |\mathring h|^2 f+\frac{1}{nc_0 -1}|\hat A|^2f  \\
 & = \frac{1}{n} |H|^2 f + |\mathring h|^2 f+\frac{1}{nc_0 -1}|\hat A|^2f \\
 & = |h|^2 f+ \frac{1}{nc_0 -1}|\hat A|^2f. 
\end{align*}
Consequently by \eqref{evolution-equation-of-f} and \eqref{for-evolution-of-f}, we have
\begin{equation*}
\Big(\frac{\partial}{\partial t}- \Delta \Big)f \geq 2|h|^2 f + \frac{2}{nc_0 -1} |\hat A|^2 f + 2(|\nabla^\perp A|^2 - c_0 |\nabla^\perp H|^2) 
\end{equation*}
Hence, by \eqref{classic-improved-derivative-estimate} 
\begin{equation*}
\Big(\frac{\partial}{\partial t}- \Delta \Big)f \geq 2|h|^2f  +\frac{2}{nc_0 -1} |\hat A|^2 f + \Big(1 - \frac{c_0(n+2)}{3}\Big) |\nabla^\perp A|^2.
\end{equation*}
This is the same inequality obtained in \cite{AB10} and \cite{Ngu18}.
\end{proof}

There is an analogue of \eqref{classic-improved-derivative-estimate} in both the principal direction and its orthogonal complement. We observed in Section 2 that the projection of the Codazzi identity onto $\nu_1$ and its orthogonal complement implies the tensors $\nabla_i h_{jk} + \langle \nabla^\perp_i \hat A_{jk}, \nu_1 \rangle$ and $\hat \nabla_i^\perp \hat A_{jk} + h_{jk} \nabla^\perp_i \nu_1$ are symmetric in $i, j, k$. Recalling \eqref{codazzi-1} and \eqref{codazzi-perp}, it follows that an irreducible component of each tensor is given by
\begin{align*}
E^{(1)}_{ijk} &:= \frac{1}{n+2}\Big(g_{ij} \nabla_k |H| + g_{jk} \nabla_i |H| +g_{ki} \nabla_j |H|\Big),\\
E^{(\perp)}_{ijk} &:= \frac{1}{n+2}\Big(g_{ij}|H| \nabla_k^\perp \nu_1 + g_{jk}|H| \nabla_i^\perp \nu_1 +g_{ki}|H| \nabla_j^\perp \nu_1\Big).
\end{align*}
You can readily confirm that $E^{(1)}_{ijk}(\nabla_i h_{jk} + \langle \nabla^\perp_i \hat A_{jk}, \nu_1 \rangle) = |E^{(1)}|^2$ and $\langle E^{(\perp)}_{ijk}, \hat \nabla_i^\perp \hat A_{jk} + h_{jk} \nabla^\perp_i \nu_1\rangle= |E^{(\perp)}|^2$. As in \eqref{classic-improved-derivative-estimate}, we obtain that
\begin{align}
\frac{3}{n+2} |\nabla |H||^2 &\leq |\nabla_i h_{jk} + \langle \nabla_i^\perp \hat A_{jk},\nu_1 \rangle |^2, \\
\label{improved-perp} \frac{3}{n+2}|H|^2 |\nabla^\perp \nu_1|^2 &\leq |\hat \nabla_i^\perp \hat A_{jk} +  h_{jk}\nabla_i^\perp \nu_1|^2.
\end{align}
Now expanding the right-hand side of both inequalities above using \eqref{term-1-derivative-of-A} and \eqref{term-perp-derivative-of-A}; recalling \eqref{traceless-derivative-term}; and noting that $\frac{3}{n+2} - \frac{1}{n} = \frac{2(n-1)}{n(n+2)}$, we arrive at the estimates
\begin{align}
\label{final-improved-1} \frac{2(n-1)}{n(n+2)}|\nabla |H||^2 & \leq  |\langle \nabla_i^\perp \mathring A_{jk}, \nu_1\rangle|^2,\\
\label{final-improved-perp}  \frac{2(n-1)}{n(n+2)} |H|^2 |\nabla^\perp \nu_1|^2 &\leq |\hat \nabla_i^\perp \hat A_{jk} + \mathring h_{jk}\nabla^\perp_i \nu_1|^2.  
\end{align}
The second of these two estimates implies the following useful lower bound. 
\begin{lemma}[Lower bound for Bochner term of $(\partial_t - \Delta)|\hat A|^2$]\label{bochA}\leavevmode 
\begin{enumerate}
\item If $\frac{1}{n} < c_0 \leq \frac{4}{3n}$, then
\begin{equation}
\label{lower-bochner-A-1} 2|\hat \nabla^\perp \hat A|^2 \geq  \frac{4n-10}{n+2}|\mathring h|^2|\nabla^\perp \nu_1|^2  +  \frac{6(n-1)}{n+2}|\hat A|^2|\nabla^\perp \nu_1|^2  +  \frac{6(n-1)}{n+2}f |\nabla^\perp \nu_1|^2.
\end{equation}
\item If $\frac{1}{n} < c_0 \leq \frac{3(n+1)}{2n(n+2)}$, then
\begin{equation}
\label{lower-bochner-A-2} 2|\hat \nabla^\perp \hat A|^2 \geq  2|\mathring h|^2|\nabla^\perp \nu_1|^2  + 4|\hat A|^2|\nabla^\perp \nu_1|^2  + 4 f |\nabla^\perp \nu_1|^2.
\end{equation}
\end{enumerate}
\end{lemma}
\begin{proof}
We begin by applying Young's inequality
\begin{align*}
|\hat \nabla_i^\perp \hat A_{jk} +  \mathring h_{jk}\nabla_i^\perp \nu_1|^2 &= |\hat \nabla^\perp \hat A|^2 + 2\langle \hat \nabla_i^\perp \hat A_{jk},  \mathring h_{jk} \nabla_i^\perp \nu_1 \rangle + |\mathring h|^2 |\nabla^\perp \nu_1|^2 \\
& \leq 2|\hat \nabla^\perp \hat A|^2 + 2|\mathring h|^2 |\nabla^\perp \nu_1|^2. 
\end{align*}
Multiplying both sides of \eqref{H-terms} by $\frac{2(n-1)}{(n+2)(nc_0 -1)}$ gives
\[
 \frac{2(n-1)}{n(n+2)} |H|^2 = \frac{2(n-1)}{(n+2)(nc_0 -1)}(f + |\hat A|^2 + |\mathring h|^2).
\]
In view of \eqref{final-improved-perp}, our observations give us that
\[
 \frac{2(n-1)}{(n+2)(nc_0 -1)}(f + |\hat A|^2 + |\mathring h|^2) |\nabla^\perp \nu_1|^2\leq 2|\hat \nabla^\perp \hat A|^2 + 2|\mathring h|^2 |\nabla^\perp \nu_1|^2.
\]
Subtracting the $|\mathring h|^2|\nabla^\perp \nu_1|^2$ term on the right-hand side gives
\[
\frac{2(n-1)}{(n+2)(nc_0 -1)}( f+ |\hat A|^2) |\nabla^\perp \nu_1|^2 +\Big(\frac{2(n-1)}{(n+2)(nc_0 -1)} - 2\Big) | \mathring h|^2|\nabla^\perp \nu_1|^2 \leq 2|\hat \nabla^\perp \hat A|^2.
\]
If $c_0 \leq \frac{4}{3n}$, then $nc_0 -1 \leq \frac{1}{3}$ and 
\[
\frac{2(n-1)}{(n+2)(nc_0 -1)} \geq \frac{6(n-1)}{n+2}. 
\]
Plugging this in above gives the first estimate of the lemma. If $c_0 \leq \frac{3(n+1)}{2n(n+2)}$, then $nc_0 -1 \leq \frac{n-1}{2(n+2)}$ and 
\[
\frac{2(n-1)}{(n+2)(nc_0 -1)} \geq 4.
\]
This establishes the second estimate in the lemma. 
\end{proof}

Next we obtain improved lower bounds for the Bochner term in the evolution equation of $f$. 
\begin{lemma}[Lower bound for Bochner term of $(\partial_t - \Delta)f$]\label{bochf} \leavevmode
\begin{enumerate}
\item If $\frac{1}{n} < c_0 \leq \frac{4}{3n}$, then
\begin{equation}
\label{lower-bochner-f-1} 2\frac{|\hat A|^2}{f} \big(|\nabla^\perp A|^2 - c_0|\nabla^\perp H|^2) \geq \frac{5n-8}{3(n-1)}\frac{|\hat A|^2}{f}|\langle \nabla^\perp \mathring A, \nu_1 \rangle|^2 + \frac{10n - 16}{n+2} |\hat A|^2 |\nabla^\perp \nu_1|^2. 
\end{equation}
\item If $\frac{1}{n} < c_0 \leq \frac{3(n+1)}{2n(n+2)}$, then 
\begin{equation}
\label{lower-bochner-f-2} 2\frac{|\hat A|^2}{f} \big(|\nabla^\perp A|^2 - c_0|\nabla^\perp H|^2) \geq \frac{3}{2}\frac{|\hat A|^2}{f}|\langle \nabla^\perp \mathring A, \nu_1 \rangle|^2 + 6 |\hat A|^2 |\nabla^\perp \nu_1|^2. 
\end{equation}
\end{enumerate}
\end{lemma} 
\begin{proof}
By \eqref{norm-of-derivative-of-A} and  \eqref{nabla-perp-H-squared-formula}, we have 
\begin{align*}
|\nabla^\perp A|^2 - c_0|\nabla^\perp H|^2 &= |\langle \nabla_i^\perp \hat A_{jk}, \nu_1\rangle + \nabla_i h_{jk} |^2 - c_0 |\nabla |H||^2 \\
& \qquad + |\hat \nabla_i^\perp \hat A_{jk} + h_{jk} \nabla_i^\perp \nu_1|^2 - c_0|H|^2|\nabla^\perp \nu_1|^2.
\end{align*}
In view of \eqref{term-1-derivative-of-A}, \eqref{traceless-derivative-term} and \eqref{final-improved-1}, we have 
\begin{align*} 
|\langle \nabla_i^\perp \hat A_{jk}, \nu_1\rangle + \nabla_i h_{jk} |^2 - c_0 |\nabla |H||^2 & = |\langle \nabla_i^\perp \mathring A_{jk}, \nu_1\rangle |^2 - \frac{nc_0 -1}{n} |\nabla |H||^2\\
& \geq  \Big(1 - \frac{(n+2)(nc_0 -1)}{2(n-1)}\Big)|\langle \nabla_i^\perp \mathring A_{jk}, \nu_1\rangle|^2.
\end{align*}
In view of \eqref{improved-perp} and \eqref{H-terms}, we have
\begin{align*}
|\hat \nabla_i^\perp \hat A_{jk} + h_{jk} \nabla_i^\perp \nu_1|^2 - c_0|H|^2|\nabla^\perp \nu_1|^2 & \geq \Big(\frac{3}{n+2} - c_0\Big)|H|^2 |\nabla^\perp \nu_1|^2 \\
&= \frac{n}{nc_0 -1}\Big(\frac{3}{n+2} - c_0 \Big) (f + |\hat A|^2 + |\mathring h|^2)|\nabla^\perp \nu_1|^2 \\
& \geq \frac{n}{nc_0 -1}\Big(\frac{3}{n+2} - c_0 \Big) f  |\nabla^\perp \nu_1|^2.
\end{align*}
Thus, by the three previous computations
\begin{align*}
2\frac{|\hat A|^2}{f} (|\nabla^\perp A|^2 - c_0|\nabla^\perp H|^2) &\geq \Big(2 - \frac{(n+2)(nc_0 -1)}{(n-1)}\Big) \frac{|\hat A|^2}{f} |\langle \nabla_i^\perp \mathring A_{jk}, \nu_1\rangle|^2\\
& \qquad +  \frac{2n}{nc_0 -1}\Big(\frac{3}{n+2} - c_0 \Big) |\hat A|^2 |\nabla^\perp \nu_1|^2.
\end{align*}
If $c_0 \leq \frac{4}{3n}$, then $nc_0 -1 \leq \frac{1}{3}$ and  
\begin{align*}
2 - \frac{(n+2)(nc_0 -1)}{(n-1)}  &\geq 2 - \frac{n+2}{3(n-1)} = \frac{5n-8}{3(n-1)},\\
\frac{2n}{nc_0 -1}\Big(\frac{3}{n+2} - c_0\Big) & \geq 6n\Big(\frac{9n - 4(n+2)}{(n+2)3n}\Big) = \frac{10n - 16}{n+2}.
\end{align*}
This establishes the first inequality of the lemma. If $c_0 \leq \frac{3(n+1)}{2n(n+2)}$, then $nc_0 -1 \leq \frac{n-1}{2(n+2)}$ and 
\begin{align*}
2 - \frac{(n+2)(nc_0 -1)}{(n-1)}  &\geq 2 - \frac{1}{2} = \frac{3}{2}, \\
\frac{2n}{nc_0 -1}\Big(\frac{3}{n+2} - c_0\Big) & \geq \frac{4n(n+2)}{n-1}\Big(\frac{6n - 3(n+1)}{2n(n+2)} \Big) = 6.
\end{align*}
This establishes the second inequality of the lemma. 
\end{proof}

Finally, we must estimate the remaining gradient term that appears in the evolution equation of $|\hat A|^2$. The term is of the form $4Q_{ijk}\langle \hat A_{ij}, \nabla_k^\perp \nu_1\rangle$ where $Q$ is defined in \eqref{definition-of-Q}. By Cauchy-Schwarz, we obtain the following useful estimate:
\begin{align}
\nonumber |\langle \hat A, \nabla^\perp \nu_1 \rangle|^2 &= \sum_{i, j,k =1}^n \langle \hat A_{ij}, \nabla_k^\perp \nu_1 \rangle^2\\
\nonumber & \leq \sum_{i,j,k= 1}^n \Big(\sum_{\beta =2}^{N-2} \hat A_{ij\beta}^2\Big) \Big(\sum_{\beta = 2}^{N-2} \langle \nabla_k^\perp \nu_1, \nu_\beta \rangle^2\Big) \\
\label{cauchy-schwarz-comp}& \leq |\hat A|^2 |\nabla^\perp \nu_1|^2. 
\end{align}

\noindent \begin{lemma}[Upper bound for gradient term of $(\partial_t - \Delta)|\hat A|^2$]\label{gradA}\leavevmode
\begin{enumerate}
\item If $\frac{1}{n} < c_0 \leq \frac{4}{3n}$, then 
\begin{align}
\label{estimate-for-gradient-term-1} 4Q_{ijk}\langle \hat A_{ij}, \nabla^\perp_k\nu_1\rangle &\leq  2|\langle \nabla^\perp \hat A, \nu_1 \rangle|^2 +   \frac{5n-9}{3(n-1)}\frac{|\hat A|^2}{f}|\langle \nabla^\perp \mathring A, \nu_1\rangle|^2\\
\nonumber & \qquad +  2|\hat A|^2|\nabla^\perp \nu_1|^2 + \frac{3(n-1)}{n-3}f|\nabla^\perp \nu_1|^2 + \frac{2(n+2)}{n+3}  |\mathring h|^2|\nabla^\perp \nu_1|^2. 
\end{align}
\item If $\frac{1}{n} < c_0 \leq \frac{3(n+1)}{2n(n+2)} - \varepsilon_0$ and $\varepsilon = \frac{2n(n+2)}{3(n-1)}\varepsilon_0$, then
\begin{align}
\label{estimate-for-gradient-term-2} 4Q_{ijk} \langle \hat A_{ij}, \nabla^\perp_k\nu_1\rangle &\leq 2|\langle \nabla^\perp \hat A, \nu_1 \rangle|^2 + (1-\varepsilon)\frac{3}{2}\frac{|\hat A|^2}{f}|\langle \nabla^\perp \mathring A, \nu_1\rangle|^2\\
\nonumber & \qquad +  2|\hat A|^2|\nabla^\perp \nu_1|^2 +  4f|\nabla^\perp \nu_1|^2 + 2 |\mathring h|^2|\nabla^\perp \nu_1|^2. 
\end{align}
\end{enumerate}
\end{lemma}

\begin{proof}
Using the triangle inequality on \eqref{definition-of-Q}, we get
\begin{equation}
\label{triangle-ineq-comp} |Q| \leq |\langle \nabla^\perp \mathring A, \nu_1 \rangle| + |\langle \nabla^\perp \hat A, \nu_1 \rangle| + |H|^{-1} |\mathring h| |\nabla |H||.
\end{equation}
We will first treat the case $\frac{1}{n} < c_0 \leq \frac{4}{3n}$. It easily follows from \eqref{definition-of-f}, the definition of $f$, that
\[
f \leq \Big(c_0 - \frac{1}{n} \Big)|H|^2 \leq \frac{1}{3n}|H|^2. 
\]
Consequently, using the estimate \eqref{final-improved-1}, we obtain
\begin{align}
\label{compu-1}\frac{|\hat A|^2}{|H|^2} |\nabla |H||^2 &\leq \frac{n(n+2)}{2(n-1)} \frac{1}{3n}\frac{|\hat A|^2}{f}|\langle \nabla^\perp \mathring A, \nu_1\rangle|^2  \\
\nonumber & =  \frac{n+2}{6(n-1)}\frac{|\hat A|^2}{f}|\langle \nabla^\perp \mathring A, \nu_1\rangle|^2. 
\end{align}
Then \eqref{cauchy-schwarz-comp} and \eqref{triangle-ineq-comp} give
\begin{align*}
4Q_{ijk} \langle \hat A_{ij}, \nabla^\perp_k\nu_1\rangle &\leq 4|Q| |\langle \hat A, \nabla^\perp \nu_1\rangle| \\
&\leq 4\Big(|\langle \nabla^\perp \mathring A, \nu_1 \rangle| + |\langle \nabla^\perp \hat A, \nu_1 \rangle| + |H|^{-1} |\mathring h| |\nabla |H||\Big)|\hat A||\nabla^\perp \nu_1|. 
\end{align*}
Now to each of these three summed terms above we apply Young's inequality with constants $a_1, a_2, a_3> 0$. Specifically, we have 
\begin{align*}
4|\langle \nabla^\perp \hat A, \nu_1 \rangle||\hat A||\nabla^\perp \nu_1| & \leq 2a_1|\langle \nabla^\perp \hat A, \nu_1 \rangle|^2 + \frac{2}{a_1} |\hat A|^2|\nabla^\perp \nu_1|^2,\\
4|\langle \nabla^\perp \mathring A, \nu_1 \rangle| |\hat A||\nabla^\perp \nu_1|& = 4|\langle \nabla^\perp \mathring A, \nu_1 \rangle| \frac{|\hat A|}{f^{\frac{1}{2}}} f^{\frac{1}{2}}|\nabla^\perp \nu_1|\\
& \leq 2a_2\frac{|\hat A|^2}{f} |\langle \nabla^\perp \mathring A, \nu_1 \rangle|^2 + \frac{2}{a_2}f |\nabla^\perp \nu_1|^2,\\
4|H|^{-1} |\mathring h| |\nabla |H|||\hat A||\nabla^\perp \nu_1| &\leq 2 a_3 \frac{|\hat A|^2}{|H|^2} |\nabla |H||^2 + \frac{2}{a_3} |\mathring h|^2|\nabla^\perp \nu_1|^2\\
& \leq 2a_3 \frac{n+2}{6(n-1)}\frac{|\hat A|^2}{f}|\langle \nabla^\perp \mathring A, \nu_1\rangle|^2+ \frac{2}{a_3} |\mathring h|^2|\nabla^\perp \nu_1|^2. 
\end{align*}
Note we used \eqref{compu-1} in the last inequality. Hence
\begin{align}
\label{inequality-for-Q-comp} 4Q_{ijk} \langle \hat A_{ij}, \nabla^\perp_k\nu_1\rangle &\leq 2a_1|\langle \nabla^\perp \hat A, \nu_1 \rangle|^2 +  \big(2a_2 + 2a_3\frac{n+2}{6(n-1)}\big)\frac{|\hat A|^2}{f}|\langle \nabla^\perp \mathring A, \nu_1\rangle|^2\\
\nonumber & \qquad +  \frac{2}{a_1} |\hat A|^2|\nabla^\perp \nu_1|^2 + \frac{2}{a_2} f|\nabla^\perp \nu_1|^2 + \frac{2}{a_3} |\mathring h|^2|\nabla^\perp \nu_1|^2. 
\end{align}
Now set 
\begin{equation*}
a_1 = 1, \qquad a_2 = \frac{2(n-3)}{3(n-1)}, \qquad a_3 = \frac{n+3}{n+2}.
\end{equation*}
In this case, 
\begin{align*}
2a_2 + 2a_3 \frac{n+2}{6(n-1)} &= \frac{4(n - 3)}{3(n-1)} + \frac{n+3}{n+2}\frac{n+2}{3(n-1)}  =  \frac{5n - 9}{3(n-1)},\\
\frac{2}{a_2} &= \frac{3(n-1)}{n-3},\\
\frac{2}{a_3} &= \frac{2(n+2)}{n+3}.
\end{align*}
Plugging these into \eqref{inequality-for-Q-comp}, we conclude 
\begin{align*}
4Q_{ijk} \langle \hat A_{ij}, \nabla^\perp_k\nu_1\rangle &\leq 2|\langle \nabla^\perp \hat A, \nu_1 \rangle|^2 +   \frac{5n-9}{3(n-1)}\frac{|\hat A|^2}{f}|\langle \nabla^\perp \mathring A, \nu_1\rangle|^2\\
& \qquad +  2|\hat A|^2|\nabla^\perp \nu_1|^2 + \frac{3(n-1)}{n-3}f|\nabla^\perp \nu_1|^2 + \frac{2(n+2)}{n+3}  |\mathring h|^2|\nabla^\perp \nu_1|^2, 
\end{align*}
as claimed. 

Now if $\frac{1}{n} < c_0 \leq \frac{3(n+1)}{2n(n+2)} -\varepsilon_0$, then $c_0 - \frac{1}{n} \leq \frac{n-1}{2n(n+2)} -\varepsilon_0$. Therefore if we take $\varepsilon = \frac{2n(n+2)}{3(n-1)} \varepsilon_0$, then
\[
c_0 -\frac{1}{n} \leq (1-3\varepsilon) \frac{n-1}{2n(n+2)}.
\]
In this case, 
\[
f \leq \Big(c_0 - \frac{1}{n} \Big)|H|^2 \leq (1-3\varepsilon)\frac{n-1}{2n(n+2)}|H|^2.
\]
Again using \eqref{final-improved-1}, it follows that 
\begin{align*}
\frac{|\hat A|^2}{|H|^2} |\nabla |H||^2 &\leq (1-3\varepsilon) \frac{n(n+2)}{2(n-1)} \frac{n-1}{2n(n+2)}\frac{|\hat A|^2}{f}|\langle \nabla^\perp \mathring A, \nu_1\rangle|^2  \\
& =  \frac{1}{4}(1-3\varepsilon)\frac{|\hat A|^2}{f}|\langle \nabla^\perp \mathring A, \nu_1\rangle|^2. 
\end{align*}
Proceeding as we did before, we obtain the inequality 
\begin{align}
\label{inequality-for-Q-comp-2} 4Q_{ijk} \langle \hat A_{ij}, \nabla^\perp_k\nu_1\rangle &\leq 2a_1|\langle \nabla^\perp \hat A, \nu_1 \rangle|^2 +  \big(2a_2 + \frac{1}{2}a_3(1-3\varepsilon)\big)\frac{|\hat A|^2}{f}|\langle \nabla^\perp \mathring A, \nu_1\rangle|^2\\
\nonumber & \qquad +  \frac{2}{a_1} |\hat A|^2|\nabla^\perp \nu_1|^2 + \frac{2}{a_2} f|\nabla^\perp \nu_1|^2 + \frac{2}{a_3} |\mathring h|^2|\nabla^\perp \nu_1|^2. 
\end{align}
Set
\[
a_1 = 1, \quad a_2 = \frac{1}{2}, \quad  a_3 = 1. 
\]
In this case, 
\begin{align*}
2a_2 + \frac{1}{2}a_3 (1-3\varepsilon) &= \frac{3}{2}(1 - \varepsilon),\\
\frac{2}{a_2} &=  4\\
\frac{2}{a_3} &= 2. 
\end{align*}
Plugging these into \eqref{inequality-for-Q-comp-2}, we get
\begin{align*}
4Q_{ijk} \langle \hat A_{ij}, \nabla^\perp_k\nu_1\rangle &\leq 2|\langle \nabla^\perp \hat A, \nu_1 \rangle|^2 +(1-\varepsilon)\frac{3}{2}\frac{|\hat A|^2}{f}|\langle \nabla^\perp \mathring A, \nu_1\rangle|^2\\
& \qquad +  2|\hat A|^2|\nabla^\perp \nu_1|^2 +  4f|\nabla^\perp \nu_1|^2 + 2 |\mathring h|^2|\nabla^\perp \nu_1|^2, 
\end{align*}
as claimed. 
\end{proof}

Finally, we combine the conclusions of Lemmas \ref{bochA}, \ref{bochf}, and \ref{gradA} to get our desired estimate. 
\begin{lemma}[Gradient term estimate]\label{gradest}
Suppose either $n \geq 8$, $\frac{1}{n} < c_0 \leq \frac{4}{3n}$ and $0 < \delta \leq \frac{1}{5n-8}$; or  $\frac{1}{n} < c_0 \leq \frac{3(n+1)}{2n(n+2)} - \varepsilon_0$, and $0 < \delta \leq \min\{\frac{1}{2}, \frac{2n(n+2)}{3(n-1)} \varepsilon_0\}$. Then, in either case, 
\[
4Q_{ijk} \langle \hat A_{ij}, \nabla^\perp_k \nu_1 \rangle \leq  2|\nabla^\perp \hat A|^2 + 2(1-\delta)\frac{|\hat A|^2}{ f}(|\nabla^\perp A|^2-c_0|\nabla^\perp H|^2). 
\]
\end{lemma} 
\begin{proof}
First suppose $n \geq 8$,  $\frac{1}{n} < c_0 \leq \frac{4}{3n}$ and $0 < \delta \leq \frac{1}{5n-8}$. Expanding $|\nabla^\perp \hat A|^2$ using \eqref{norm-of-derivative-of-hat-A} and using the inequality \eqref{lower-bochner-A-1} in Lemma \ref{bochA} gives us 
\begin{align*}
2|\nabla^\perp \hat A|^2 & = 2|\hat \nabla^\perp \hat A|^2  + 2|\langle \nabla^\perp \hat A, \nu_1\rangle|^2 \\
& \geq 2|\langle \nabla^\perp \hat A, \nu_1\rangle|^2 + \frac{4n-10}{n+2}|\mathring h|^2|\nabla^\perp \nu_1|^2  +  \frac{6(n-1)}{n+2}|\hat A|^2|\nabla^\perp \nu_1|^2  +  \frac{6(n-1)}{n+2}f |\nabla^\perp \nu_1|^2.
\end{align*}
Multiplying inequality \eqref{lower-bochner-f-1} in Lemma \ref{bochf} by $(1-\delta)$ and using that $1-\delta \geq \frac{1}{2}$ on the coefficient of $|\hat A|^2 |\nabla^\perp \nu_1|^2$ gives 
\begin{align*}
2(1-\delta)\frac{|\hat A|^2}{f} \big(|\nabla^\perp A|^2 - c_0|\nabla^\perp H|^2) &\geq (1-\delta)\frac{5n-8}{3(n-1)}\frac{|\hat A|^2}{f}|\langle \nabla^\perp \mathring A, \nu_1 \rangle|^2 + \frac{5n - 8}{n+2} |\hat A|^2 |\nabla^\perp \nu_1|^2.
\end{align*}
Putting these together, we get 
\begin{align*}
2|\nabla^\perp \hat A|&^2 + 2(1-\delta)\frac{|\hat A|^2}{ f}(|\nabla^\perp A|^2-c_0|\nabla^\perp H|^2)   \\
& \geq 2|\langle \nabla^\perp \hat A, \nu_1\rangle|^2 +  (1-\delta)\frac{5n-8}{3(n-1)}\frac{|\hat A|^2}{f}|\langle \nabla^\perp \mathring A, \nu_1 \rangle|^2 \\
& \quad + \frac{11n - 14}{n+2}|\hat A|^2|\nabla^\perp \nu_1|^2  +  \frac{6(n-1)}{(n+2)} f|\nabla^\perp \nu_1|^2+  \frac{4n-10}{n+2}|\mathring h|^2|\nabla^\perp \nu_1|^2.
\end{align*}
On the other hand, the first estimate \eqref{estimate-for-gradient-term-1} of Lemma \ref{gradA} gives us that 
\begin{align*}
4Q_{ijk} \langle \hat A_{ij}, \nabla^\perp_k\nu_1\rangle &\leq 2|\langle \nabla^\perp \hat A, \nu_1 \rangle|^2 +   \frac{5n-9}{3(n-1)}\frac{|\hat A|^2}{f}|\langle \nabla^\perp \mathring A, \nu_1\rangle|^2\\
& \qquad +  2|\hat A|^2|\nabla^\perp \nu_1|^2 + \frac{3(n-1)}{n-3}f|\nabla^\perp \nu_1|^2 + \frac{2(n+2)}{n+3}  |\mathring h|^2|\nabla^\perp \nu_1|^2.
\end{align*}
Therefore, it only remains to compare the coefficients of like terms in the two inequalities above. For the coefficients of $\frac{|\hat A|^2}{f}|\langle\nabla^\perp \mathring A, \nu_1 \rangle|^2$, we need 
\[
\frac{5n-9}{3(n-1)} \leq (1-\delta)\frac{5n-8}{3(n-1)} \iff \delta \leq \frac{1}{5n - 8}.
\]
Comparing the coefficients of the remaining terms implies we need
\begin{align*}
2n + 4 &\leq 11n - 14 \iff 2 \leq n, \\
n+2 &\leq 2(n-3)\iff 8 \leq n,\\
2(n+2)^2 &\leq (4n-10)(n+3) \iff 19 \leq n(n-3).
\end{align*}
Each of these inequalities is true if $n \geq 8$ completing the proof for the first case. 

Now suppose $\frac{1}{n} < c_0 < \frac{3(n+1)}{2n(n+2)}$ and $0 < \delta < \min\{\frac{1}{2}, \frac{2n(n+2)}{3(n-1)} \varepsilon_0\}$.  Arguing as before, this time using   \eqref{lower-bochner-A-2} in Lemma \ref{bochA} and \eqref{lower-bochner-f-2} in Lemma \ref{bochf} yields
\begin{align*}
2|\nabla^\perp \hat A|&^2 + 2(1-\delta)\frac{|\hat A|^2}{ f}(|\nabla^\perp A|^2-c_0|\nabla^\perp H|^2)   \\
& \geq 2|\langle \nabla^\perp \hat A, \nu_1\rangle|^2 +  (1-\delta)\frac{3}{2}\frac{|\hat A|^2}{f}|\langle \nabla^\perp \mathring A, \nu_1 \rangle|^2 \\
& \quad + 7|\hat A|^2|\nabla^\perp \nu_1|^2  + 4 f |\nabla^\perp \nu_1|^2 + 2|\mathring h|^2|\nabla^\perp \nu_1|^2.
\end{align*}
Note we again used $\delta \leq \frac{1}{2}$ to simplify the coefficient of $|\hat A|^2 |\nabla^\perp \nu_1|^2$. On the other hand, by \eqref{estimate-for-gradient-term-2}, we have 
\begin{align*}
4Q_{ijk} \langle \hat A_{ij}, \nabla^\perp_k\nu_1\rangle &\leq 2|\langle \nabla^\perp \hat A, \nu_1 \rangle|^2 + (1-\varepsilon)\frac{3}{2}\frac{|\hat A|^2}{f}|\langle \nabla^\perp \mathring A, \nu_1\rangle|^2\\
& \qquad +  2|\hat A|^2|\nabla^\perp \nu_1|^2 +  4f|\nabla^\perp \nu_1|^2 + 2 |\mathring h|^2|\nabla^\perp \nu_1|^2,
\end{align*}
where recall $\varepsilon = \frac{2n(n+2)}{3(n-1)} \varepsilon_0$. By assumption, $\delta \leq \varepsilon$, and this completes the proof of the lemma. 
\end{proof}

\subsection{Concluding Argument} 
We now complete the proof of Theorem 1.1. Let $\delta$ be sufficiently small so that each of our above lemmas holds. Taking $\delta = \min\{\frac{1}{5n-8},\frac{2n(n+2)}{3(n-1)}\varepsilon_0\}$ suffices.  We begin by splitting off the desired nonpositive term in the evolution equation
 \begin{align}
 \nonumber \Big(\frac{\partial}{\partial t} - \Delta\Big)\frac{|\hat A|^2}{f} & = \frac{1}{f}\Big(\frac{\partial}{\partial t} - \Delta\Big)|\hat A|^2 - |\hat A|^2\frac{1}{ f^2}\Big(\frac{\partial}{\partial t} -\Delta\Big)f + 2 \Big\langle \nabla \frac{|\hat A|^2}{f}, \nabla \log  f \Big\rangle\\
\label{comp-final-1}  & = 2 \Big\langle \nabla \frac{|\hat A|^2}{f}, \nabla \log  f \Big\rangle - \delta \frac{|\hat A|^2}{ f^2}\Big(\frac{\partial}{\partial t} -\Delta\Big)f  \\
\nonumber & \qquad +  \frac{1}{f}\Big(\frac{\partial}{\partial t} - \Delta\Big)|\hat A|^2 - (1-\delta) \frac{|\hat A|^2}{ f^2}\Big(\frac{\partial}{\partial t} -\Delta\Big)f.
\end{align}
By the inequalities of Lemmas \ref{reactionest} and \ref{gradest}, the sum of terms on the second line of \eqref{comp-final-1} are nonpositive:
\begin{align*}
 \frac{1}{f}\Big(\frac{\partial}{\partial t} - \Delta\Big)|\hat A|^2& - (1-\delta) \frac{|\hat A|^2}{ f^2}\Big(\frac{\partial}{\partial t} -\Delta\Big)f \\
&  = \frac{1}{ f} \Big(2|\langle \hat A_{ij}, \hat A_{kl} \rangle|^2 + 2|\hat R^\perp|^2 + 2|R^\perp_{ij}(\nu_1)|^2 \Big)\\
& \quad - \frac{1}{f}\Big(2(1-\delta)\frac{|\hat A|^2}{ f} (c_0|\langle A_{ij}, H \rangle|^2 - |\langle A_{ij}, A_{kl} \rangle|^2 - |R^\perp_{ij}|^2)\Big)\\
& \quad + \frac{1}{ f} \Big(4Q_{ijk} \langle \hat A_{ij}, \nabla^\perp_k \nu_1 \rangle -  2|\nabla^\perp \hat A|^2 - 2(1-\delta)\frac{|\hat A|^2}{ f}(|\nabla^\perp A|^2-c_0|\nabla^\perp H|^2)\Big) \\
& \leq 0.
\end{align*}
Thus we have finally obtained our initial claim \eqref{goal} at the beginning of this section:
\begin{equation}
\label{goal-again}\Big(\frac{\partial}{\partial t} - \Delta\Big)\frac{|\hat A|^2}{f} \leq  2 \Big\langle \nabla \frac{|\hat A|^2}{f}, \nabla \log f\Big\rangle - \delta \frac{|\hat A|^2}{f^2} \Big(\frac{\partial}{\partial t} - \Delta \Big)f.
\end{equation}
Recall that $(\frac{\partial}{\partial t}-\Delta) f$ is nonnegative at each point in space-time.
Let $\sigma = \delta$. We compute that
\begin{align*}
\Big(\frac{\partial}{\partial t} - \Delta\Big)f^{1-\sigma} & = (1-\sigma) f^{-\sigma} \Big(\frac{\partial}{\partial t} - \Delta\Big)f + \sigma (1-\sigma)f^{-\sigma - 1} |\nabla f|^2\\
& \geq  (1-\sigma) f^{-\sigma} \Big(\frac{\partial}{\partial t} - \Delta\Big)f.
\end{align*} 
Then, making use of \eqref{useful-identity}, we have
\begin{align*}
\Big(\frac{\partial}{\partial t} - \Delta\Big)\frac{|\hat A|^2}{f^{1-\sigma}} & = \frac{1}{f^{1-\sigma}}\Big(\frac{\partial}{\partial t} - \Delta\Big)|\hat A|^2 - |\hat A|^2\frac{1}{ f^{2-2\sigma}}\Big(\frac{\partial}{\partial t} -\Delta\Big)f^{1-\sigma} + 2 \Big\langle \nabla \frac{|\hat A|^2}{f^{1-\sigma}}, \nabla \log  f^{1-\sigma} \Big\rangle\\
&\leq   \frac{1}{f^{1-\sigma}}\Big(\frac{\partial}{\partial t} - \Delta\Big)|\hat A|^2 - |\hat A|^2\frac{1}{ f^{2-2\sigma}}(1-\sigma) f^{-\sigma} \Big(\frac{\partial}{\partial t} -\Delta\Big)f \\
&  \qquad + 2 \Big\langle \nabla \frac{|\hat A|^2}{f^{1-\sigma}}, \nabla \log  f^{1-\sigma} \Big\rangle\\ 
& = f^{\sigma}\Big(\frac{1}{f}\Big(\frac{\partial}{\partial t} - \Delta\Big)|\hat A|^2 - \frac{|\hat A|^2}{f^2}\Big(\frac{\partial}{\partial t} -\Delta\Big)f\Big) +\sigma \frac{|\hat A|^2}{f^2}  f^{\sigma}\Big(\frac{\partial}{\partial t} -\Delta\Big)f\\
& \qquad + 2 \Big\langle \nabla \frac{|\hat A|^2}{f^{1-\sigma}}, \nabla \log  f^{1-\sigma} \Big\rangle.
\end{align*}
Now again by \eqref{useful-identity} and \eqref{goal-again}
\begin{align*}
\frac{1}{f}\Big(\frac{\partial}{\partial t} - \Delta\Big)|\hat A|^2 - \frac{|\hat A|^2}{f^2}\Big(\frac{\partial}{\partial t} -\Delta\Big)f  & = \Big(\frac{\partial}{\partial t} - \Delta \Big)\frac{|\hat A|^2}{f} - 2\Big\langle \nabla \frac{|\hat A|^2}{f}, \nabla \log f \Big\rangle \\
& \leq -\delta \frac{|\hat A|^2}{f^2} \Big(\frac{\partial}{\partial t} - \Delta\Big)f.
\end{align*}
Therefore, since $\sigma = \delta$,
\begin{align*}
\Big(\frac{\partial}{\partial t} - \Delta\Big)\frac{|\hat A|^2}{f^{1-\sigma}} & \leq- \delta \frac{|\hat A|^2}{f^2} f^{\sigma}  \Big(\frac{\partial}{\partial t} - \Delta \Big)f +\sigma \frac{|\hat A|^2}{f^2}  f^{\sigma}\Big(\frac{\partial}{\partial t} -\Delta\Big)f   \\
& \qquad + 2 \Big\langle \nabla \frac{|\hat A|^2}{f^{1-\sigma}}, \nabla \log  f^{1-\sigma} \Big\rangle\\
& = 2 \Big\langle \nabla \frac{|\hat A|^2}{f^{1-\sigma}}, \nabla \log  f^{1-\sigma} \Big\rangle.
 \end{align*}
Hence by the maximum principle, there exists a constant $C$ depending only upon the initial manifold $M_0$ such that $|\hat A|^2 \leq Cf^{1-\sigma}$ on $M \times [0, T)$. Since $f \leq c_0 |H|^2$, this implies $|\hat A|^2 \leq  C |H|^{2-2\sigma}$. This completes the proof of Theorem \ref{main}.

\sc{Department of Mathematics, Columbia University, New York, NY 10027}

\end{document}